\documentclass[12pt]{article}

\usepackage{amssymb}
\usepackage{amsmath,amsthm}
\usepackage[latin1]{inputenc}
\usepackage{hyperref}
\usepackage{enumerate}
\usepackage{color}
\usepackage{graphicx}
\usepackage{epstopdf}
\usepackage{tikz}
\hypersetup{colorlinks=true, linkcolor=blue, citecolor=blue,urlcolor=blue}

 \newtheorem{remark}{Remark}

 \newtheorem{lemma}[remark]{Lemma}
 \newtheorem{observation}[remark]{Observation}
 \newtheorem{theorem}[remark]{Theorem}
 \newtheorem{proposition}[remark]{Proposition}
 \newtheorem{corollary}[remark]{Corollary}

  \newtheorem{prelem}{{\bf Theorem}}

\title{On the Strong Roman Domination Number of Graphs}

\author{$^{(1)}$M. P. \'Alvarez-Ruiz, $^{(2)}$I. Gonz\'alez Yero, $^{(1)}$T. Mediavilla-Gradolph,\\  $^{(3)}$S. M. Sheikholeslami, and $^{(2)}$J. C. Valenzuela-Tripodoro\\
\\
$^{(1)}${\small Departamento de Estad\'istica e Investigaci\'on Operativa}\\
$^{(2)}${\small Departamento de Matem\'aticas}\\
{\small Escuela Polit\'ecnica Superior de Algeciras, Universidad de C\'adiz}\\ {\small
Av. Ram\'on Puyol s/n, 11202 Algeciras, Spain} \\
{\tt\small \{pilar.ruiz; ismael.gonzalez; teresa.mediavilla; jcarlos.valenzuela\}@uca.es}\\
$^{(3)}${\small  Department of Mathematics} \\
{\small Azarbaijan Shahid Madani University}\\
{\small Tabriz, I. R. Iran}\\
{\tt\small s.m.sheikholeslami@azaruniv.edu} }

\date{}

\begin{document}

\maketitle
\begin{abstract}
Based on the history that the Emperor Constantine decreed that any
undefended place (with no legions) of the Roman Empire must be
protected by a ``stronger'' neighbor place (having two legions), a
graph theoretical model called Roman domination in graphs was
described. A Roman dominating function for a graph $G=(V,E)$, is a
function $f:V\rightarrow \{0,1,2\}$ such that every vertex $v$
with $f(v)=0$ has at least a neighbor $w$ in $G$ for which
$f(w)=2$. The Roman domination number of a graph is the minimum
weight, $\sum_{v\in V}f(v)$, of a Roman dominating function.
 In this paper we initiate the study of a new parameter related to
Roman domination, which we call strong Roman domination number and
denote it by $\gamma_{StR}(G)$.  We approach the
problem of a Roman domination-type defensive strategy under
multiple simultaneous attacks and begin with the study of several
mathematical properties of this invariant.  In
particular, we first show that the decision problem regarding the
computation of the strong Roman domination number is NP-complete,
even when restricted to bipartite graphs. We obtain several bounds on such a parameter and give some realizability results for it. Moreover, we prove that
for any tree $T$ of order $n\ge 3$, $\gamma_{StR}(T)\le 6n/7$ and characterize all extremal trees. \\


\end{abstract}

\section{Introduction}

The concept of Roman domination in graphs was introduced by
Cockayne \emph{et al.} \cite{CDHH04}, according to some
connections with historical problems of defending the Roman Empire
described in \cite{rr,S99}. A {\em Roman dominating function}
(RDF for short) on a graph $G = (V ,E)$ is defined as a function $f : V
\longrightarrow\{0, 1, 2\}$ satisfying the condition that every
vertex $v$ for which $f(v) = 0$ is adjacent to at least one vertex
$u$ for which $f(u)= 2$. The {\em weight} of an RDF $f$ is the
value $\omega(f)=\sum_{v\in V}f (v)$. The {\em Roman domination
number} of a graph $G$, denoted by $\gamma_R(G)$, equals the
minimum weight of an RDF on $G$. A $\gamma_R(G)$-{\em function} is a
Roman dominating function of $G$ with weight $\gamma_R(G)$. After
this seminal work \cite{CDHH04}, several investigations have been
focused into obtaining properties of this invariant
\cite{D00,FKKS09,H02,H03,XYB09}.

On the other hand, in order to generalize or improve some
particular property of the Roman domination in its standard
presentation,  some variants  of Roman
domination have been introduced and frequently studied. Those
variants are frequently related to modifying the conditions in
which the vertices are dominated, or to adding an extra property to
the Roman domination property itself. For instance we remark here
variants like the following ones: independent Roman domination
\cite{indep-rom,indep-rom-2}, edge Roman domination
\cite{edge-rom}, weak Roman domination \cite{weak-rom,weak-rom-2},
total Roman domination\footnote{The concept of total Roman domination was introduced in \cite{total-rom} albeit in a more general setting. Its specific definition has appeared in \cite{yero}.} \cite{total-rom}, signed Roman domination
\cite{signed-rom,sv}, signed Roman edge domination \cite{AA},
Roman $k$-domination \cite{k-roman-2,k-roman} and distance Roman
domination \cite{distance-rom}, among others. On the other hand, an
interesting version regarding the defense of the ``Roman Empire''
against multiple attacks was described in \cite{H03}. In this
article we propose a new version of Roman domination in which we also
deal with multiple attacks.

To begin with our work, we first introduce the terminology and
notation we shall use throughout the exposition. Unless stated on
the contrary, other notation and terminology not explicitly given
here could be find in \cite{CL05}.  Let $G$ be a
simple graph with vertex set $V=V(G)$ and edge set $E=E(G)$. The
\emph{order} $|V|$ of $G$ is denoted by $n=n(G)$ and the \emph{size} $|E|$ of
$G$ is denoted by $m=m(G)$. By $u\sim v$ we mean that $u,v$ are
adjacent, \emph{i.e.}, $uv\in E$. For a non-empty set $X\subseteq
V$, and a vertex $v\in V$, $N_X(v)$ denotes the set of neighbors
$v$ has in $X$, or equivalently, $N_X(v)= \{u\in X: u\sim v\}$. In
the case $X=V$, we use only $N(v)$, instead of $N_V(v)$, which is
also called the \emph{open neighborhood} of the vertex $v\in V$.
The \emph{close neighborhood} of a vertex $v\in V$ is
$N[v]=N(v)\cup \{v\}$. For any vertex $v$, the cardinality of
$N(v)$ is the \emph{degree} of $v$ in $G$, denoted by $\deg_G(v)$
(or just $\deg(v)$ if there is no confusion). The {\em minimum}
and {\em maximum degree} of a graph $G$ are denoted by
$\delta=\delta(G)$ and $\Delta=\Delta(G)$, respectively. The {\em
open neighborhood} of a set $S\subseteq V$ is the set
$N(S)=\cup_{v\in S}N(v)$, and the {\em closed neighborhood} of $S$
is the set $N[S]=N(S)\cup S$. A \emph{universal vertex} of $G$ is
a vertex which is adjacent to every other vertex of $G$.

A $uv$-\emph{path} in $G$, joining the (end) vertices $u,v\in V$,
is a finite alternating sequence: $u_0=u,e_1,u_1,e_2, \ldots ,
u_{k-1}, e_k,u_k=v$ of different vertices and edges, beginning
with the vertex $u$ and ending with the vertex $v$, so that
$e_i=u_{i-1}u_i$ for all $i=1,2,\ldots,k$. The number of edges in
a path is called the \emph{length} of the path. The length of a
shortest $uv$-path is the \emph{distance between the vertices} $u$
and $v$, and it is denoted by $d(u,v)$. The maximum among all the
distances between two vertices in a graph $G$ is denoted by
$Diam(G)$, the \emph{diameter} of $G$. A \emph{cycle} is a
$uu$-path. The {\em girth} of a graph $G$, denoted by $g(G)$, is
the length of its shortest cycle. The girth of a graph with no
cycle is defined $\infty$.

The set of vertices $D\subset V$ is a {\em dominating set} if
every vertex $v$ not in $D$ is adjacent to at least one vertex in
$D$. The minimum cardinality of any dominating set of $G$ is the
{\em domination number} of $G$ and is denoted by $\gamma(G)$. A
dominating set $D$ in $G$ with $|D|=\gamma(G)$ is called a
$\gamma(G)$-set. Notice that a graph having a universal vertex has
domination number equal to one.

Let $f$ be a Roman dominating function for $G$ and let
 $V(G)=B_0\cup B_1\cup B_2$  be the sets of vertices of $G$
induced by $f$, where $B_i = \{v\in V\;:\; f(v) = i\}$, for all
$i\in \{0, 1, 2\}$. It is clear that for any Roman dominating
function $f$ of a graph $G$, we have that $f(V)=\sum_{u\in
V}f(u)=2|B_2|+|B_1|$.   A Roman dominating
function $f$ can be represented by the ordered partition
$(B_0,B_1,B_2)$ of $V(G)$. It is proved that for any graph $G$,
$\gamma(G)\le \gamma_R(G)\le 2\gamma(G)$ \cite{CDHH04}. Note that
if $C_1,C_2,\ldots,C_t$ are the components of $G$, then
$\gamma_R(G)= \sum_{i=1}^t\gamma_R(C_i)$.  Therefore,
from now on we will only consider connected graphs, unless it
would be necessary to satisfy some specific condition.

The defensive strategy of Roman domination is based in the fact that every place in which
there is established a Roman legion (a label 1 in the Roman dominating function) is able to protect itself under external attacks; and
that every unsecured place (a label 0) must have at least a stronger neighbor (a label 2).
In that way, if an unsecured place (a label 0) is attacked, then a stronger neighbor could send one
of its two legions in order to defend the weak neighbor vertex (label 0)
from the attack. Two examples of Roman dominating functions are depicted in Figure \ref{fig-Romanfs}.

\begin{figure}[ht]
\centering
  \begin{tikzpicture}[scale=.7, transform shape]
\node [draw, shape=circle, fill=black] (1) at  (0,0) {};
\node at (-0.6,-0.6) {1};
\node [draw, shape=circle, fill=black] (2) at  (0,1.5) {};
\node at (-0.6,1.5) {0};
\node [draw, shape=circle, fill=black] (3) at  (0,3) {};
\node at (-0.6,3) {0};
\node [draw, shape=circle, fill=black] (4) at  (0,4.5) {};
\node at (-0.6,4.5) {0};
\node [draw, shape=circle, fill=black] (5) at  (1.5,0) {};
\node at (1.5,-0.6) {0};
\node [draw, shape=circle, fill=black] (6) at  (1.5,6) {};
\node at (1.5,6.6) {0};
\node [draw, shape=circle, fill=black] (7) at  (3,0) {};
\node at (3,-0.6) {0};
\node [draw, shape=circle, fill=black] (8) at  (3,3) {};
\node at (3.6,3.6) {2};
\node [draw, shape=circle, fill=black] (9) at  (3,6) {};
\node at (3,6.6) {0};
\node [draw, shape=circle, fill=black] (10) at  (4.5,0) {};
\node at (4.5,-0.6) {0};
\node [draw, shape=circle, fill=black] (11) at  (4.5,6) {};
\node at (4.5,6.6) {0};
\node [draw, shape=circle, fill=black] (12) at  (6,1.5) {};
\node at (6.6,1.5) {0};
\node [draw, shape=circle, fill=black] (13) at  (6,3) {};
\node at (6.6,3) {0};
\node [draw, shape=circle, fill=black] (14) at  (6,4.5) {};
\node at (6.6,4.5) {0};

\node [draw, shape=circle, fill=black] (15) at  (9,0) {};
\node at (8.4,-0.6) {0};
\node [draw, shape=circle, fill=black] (16) at  (9,2) {};
\node at (8.4,2) {1};
\node [draw, shape=circle, fill=black] (17) at  (9,4) {};
\node at (8.4,4) {0};
\node [draw, shape=circle, fill=black] (18) at  (9,6) {};
\node at (8.4,6.6) {0};
\node [draw, shape=circle, fill=black] (19) at  (10.5,0) {};
\node at (10.5,-0.6) {0};
\node [draw, shape=circle, fill=black] (20) at  (10.5,2) {};
\node at (9.9,2) {2};
\node [draw, shape=circle, fill=black] (21) at  (10.5,4) {};
\node at (10.5,4.6) {2};
\node [draw, shape=circle, fill=black] (22) at  (12,0) {};
\node at (12,-0.6) {0};
\node [draw, shape=circle, fill=black] (23) at  (12,2) {};
\node at (12,2.6) {2};
\node [draw, shape=circle, fill=black] (24) at  (12,4) {};
\node at (12,4.6) {2};
\node [draw, shape=circle, fill=black] (25) at  (13.5,0) {};
\node at (14.1,-0.6) {0};
\node [draw, shape=circle, fill=black] (26) at  (13.5,2) {};
\node at (14.1,2) {0};
\node [draw, shape=circle, fill=black] (27) at  (13.5,4) {};
\node at (14.1,4) {0};
\node [draw, shape=circle, fill=black] (28) at  (13.5,6) {};
\node at (14.1,6.6) {0};

\draw(8)--(2)--(1)--(1)--(5)--(8)--(3);
\draw(8)--(4);
\draw(8)--(6);
\draw(8)--(7);
\draw(8)--(9);
\draw(8)--(10);
\draw(8)--(11);
\draw(8)--(12);
\draw(8)--(13);
\draw(8)--(14);

\draw(18)--(21)--(17)--(16)--(15)--(20)--(19);
\draw(20)--(21)--(23)--(22);
\draw(21)--(24)--(27);
\draw(25)--(23)--(26);
\draw(24)--(28);
\end{tikzpicture}
\caption{Two Roman dominating functions.}\label{fig-Romanfs}
\end{figure}
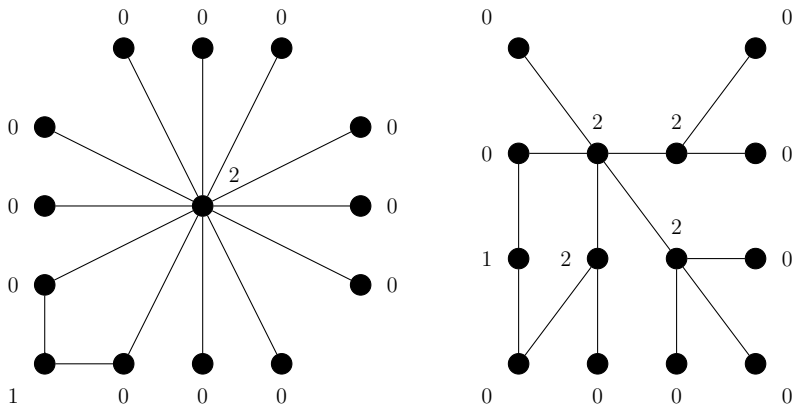

Although these two functions (Figure \ref{fig-Romanfs}) satisfy
the conditions to be Roman dominating functions, they correspond
to two very different real situations. The unique strong place (2) in
the left hand side graph must defend up to 12 weak places from
possible external attacks. However, in the right hand side graph,
the task of defending the unsecured places is divided into several
strong places. This observation has led us to pose the following
question: how many weak places may defend a strong place having
two legions? Taking into account that the strong place must leave
one of its legions to defend itself, the situation depicted on the
left hand side graph of Figure \ref{fig-Romanfs} seems to be a not
efficient defensive strategy: the Roman domination strategy fails
against a ``multiple attack'' situation. If several simultaneous
attacks to weak places are developed, then the only stronger place
will be not able to defend its neighbors efficiently. With this
motivation in mind, we introduce the concept of strong Roman
dominating function as follows. For our purposes, we consider that
a strong place should be able to defend itself and, at least half
of its weak neighbors.

Consider a graph $G$ of order $n$ and maximum degree $\Delta$. Let
$f:V(G)\rightarrow \{0,1,\ldots , \left\lceil
\frac{\Delta}{2}\right\rceil+1\}$ be a function that labels the vertices
of $G$. Let $B_j=\{v\in V:f(v)=j\}$ for $j=0,1$ and let
$B_2=V\setminus (B_0\cup B_1)=\{v\in V: f(v)\ge 2 \}$. Then, $f$
is a \emph{strong Roman dominating function} (StRDF for short) for $G$,  if every $v\in B_0$ has
a neighbor $w$, such that $w\in B_2$ and $f(w)\ge 1+ \left\lceil
\frac{1}{2} |N(w)\cap B_0| \right\rceil.$  In Figure
\ref{fig.strongRomanfs} we show a strong Roman dominating function
for each of the graphs shown in Figure \ref{fig-Romanfs}. The
minimum weight, $w(f)=f(V) = \sum_{u\in V}f(u)$, over all the
strong Roman dominating functions for $G$, is called the
\emph{strong Roman domination number} of $G$ and we denote it by
$\gamma_{StR}(G)$. An StRDF of minimum weight is called a
$\gamma_{StR}(G)$-function.

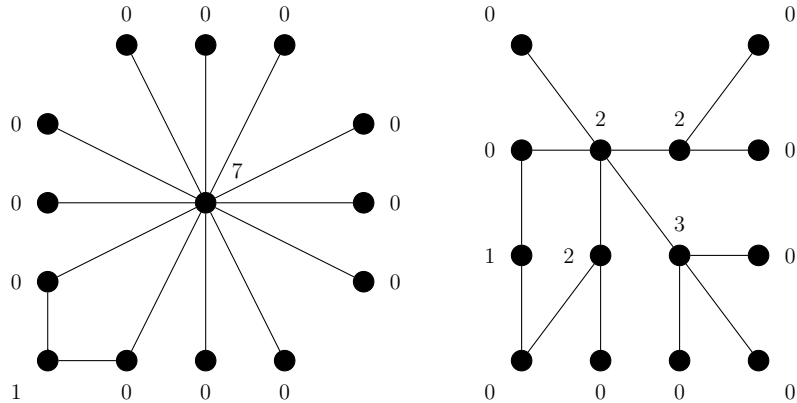
\begin{figure}[ht]
\centering
\begin{tikzpicture}[scale=.7, transform shape]
\node [draw, shape=circle, fill=black] (1) at  (0,0) {};
\node at (-0.6,-0.6) {1};
\node [draw, shape=circle, fill=black] (2) at  (0,1.5) {};
\node at (-0.6,1.5) {0};
\node [draw, shape=circle, fill=black] (3) at  (0,3) {};
\node at (-0.6,3) {0};
\node [draw, shape=circle, fill=black] (4) at  (0,4.5) {};
\node at (-0.6,4.5) {0};
\node [draw, shape=circle, fill=black] (5) at  (1.5,0) {};
\node at (1.5,-0.6) {0};
\node [draw, shape=circle, fill=black] (6) at  (1.5,6) {};
\node at (1.5,6.6) {0};
\node [draw, shape=circle, fill=black] (7) at  (3,0) {};
\node at (3,-0.6) {0};
\node [draw, shape=circle, fill=black] (8) at  (3,3) {};
\node at (3.6,3.6) {7};
\node [draw, shape=circle, fill=black] (9) at  (3,6) {};
\node at (3,6.6) {0};
\node [draw, shape=circle, fill=black] (10) at  (4.5,0) {};
\node at (4.5,-0.6) {0};
\node [draw, shape=circle, fill=black] (11) at  (4.5,6) {};
\node at (4.5,6.6) {0};
\node [draw, shape=circle, fill=black] (12) at  (6,1.5) {};
\node at (6.6,1.5) {0};
\node [draw, shape=circle, fill=black] (13) at  (6,3) {};
\node at (6.6,3) {0};
\node [draw, shape=circle, fill=black] (14) at  (6,4.5) {};
\node at (6.6,4.5) {0};

\node [draw, shape=circle, fill=black] (15) at  (9,0) {};
\node at (8.4,-0.6) {0};
\node [draw, shape=circle, fill=black] (16) at  (9,2) {};
\node at (8.4,2) {1};
\node [draw, shape=circle, fill=black] (17) at  (9,4) {};
\node at (8.4,4) {0};
\node [draw, shape=circle, fill=black] (18) at  (9,6) {};
\node at (8.4,6.6) {0};
\node [draw, shape=circle, fill=black] (19) at  (10.5,0) {};
\node at (10.5,-0.6) {0};
\node [draw, shape=circle, fill=black] (20) at  (10.5,2) {};
\node at (9.9,2) {2};
\node [draw, shape=circle, fill=black] (21) at  (10.5,4) {};
\node at (10.5,4.6) {2};
\node [draw, shape=circle, fill=black] (22) at  (12,0) {};
\node at (12,-0.6) {0};
\node [draw, shape=circle, fill=black] (23) at  (12,2) {};
\node at (12,2.6) {3};
\node [draw, shape=circle, fill=black] (24) at  (12,4) {};
\node at (12,4.6) {2};
\node [draw, shape=circle, fill=black] (25) at  (13.5,0) {};
\node at (14.1,-0.6) {0};
\node [draw, shape=circle, fill=black] (26) at  (13.5,2) {};
\node at (14.1,2) {0};
\node [draw, shape=circle, fill=black] (27) at  (13.5,4) {};
\node at (14.1,4) {0};
\node [draw, shape=circle, fill=black] (28) at  (13.5,6) {};
\node at (14.1,6.6) {0};

\draw(8)--(2)--(1)--(1)--(5)--(8)--(3);
\draw(8)--(4);
\draw(8)--(6);
\draw(8)--(7);
\draw(8)--(9);
\draw(8)--(10);
\draw(8)--(11);
\draw(8)--(12);
\draw(8)--(13);
\draw(8)--(14);

\draw(18)--(21)--(17)--(16)--(15)--(20)--(19);
\draw(20)--(21)--(23)--(22);
\draw(21)--(24)--(27);
\draw(25)--(23)--(26);
\draw(24)--(28);
\end{tikzpicture}
\caption{Two strong Roman dominating functions.}
\label{fig.strongRomanfs}
\end{figure}

The relationship between the vertices having label two (2) in a Roman dominating function,
and those ones having labels with value greater than one in a strong
Roman dominating function is not exactly clear, as we can observe throughout the following examples.
For instance, a minimum Roman dominating function is shown on the left hand side of Figure
\ref{fig.Romannominstrfs}. However, if we modify the labels with value two (2) to labels with value four (4),
then a strong Roman dominating function is obtained, but it has not minimum weight.
The left hand side of Figure \ref{fig.strnominromfs} shows a minimum strong Roman dominating function.
Nevertheless, if the vertices with label three (3) are changed to a label with value two (2),
then a Roman dominating function is obtained, but it has not minimum weight.

\begin{figure}[ht]
\centering
\begin{tikzpicture}[scale=.7, transform shape]
\node [draw, shape=circle, fill=black] (1) at  (0,0) {};
\node at (-0.6,-0.6) {0};
\node [draw, shape=circle, fill=black] (2) at  (0,3) {};
\node at (-0.6,3.6) {0};
\node [draw, shape=circle, fill=black] (3) at  (1.5,0) {};
\node at (1.5,-0.6) {0};
\node [draw, shape=circle, fill=black] (4) at  (1.5,1.5) {};
\node at (0.9,1.5) {2};
\node [draw, shape=circle, fill=black] (5) at  (1.5,3) {};
\node at (1.5,3.6) {0};
\node [draw, shape=circle, fill=black] (6) at  (3,1.5) {};
\node at (3,2.1) {0};
\node [draw, shape=circle, fill=black] (7) at  (4.5,0) {};
\node at (4.5,-0.6) {0};
\node [draw, shape=circle, fill=black] (8) at  (4.5,1.5) {};
\node at (5.1,1.5) {2};
\node [draw, shape=circle, fill=black] (9) at  (4.5,3) {};
\node at (4.5,3.6) {0};
\node [draw, shape=circle, fill=black] (10) at  (6,0) {};
\node at (6.6,-0.6) {0};
\node [draw, shape=circle, fill=black] (11) at  (6,3) {};
\node at (6.6,3.6) {0};

\node [draw, shape=circle, fill=black] (12) at  (9,0) {};
\node at (8.4,-0.6) {0};
\node [draw, shape=circle, fill=black] (13) at  (9,3) {};
\node at (8.4,3.6) {0};
\node [draw, shape=circle, fill=black] (14) at  (10.5,0) {};
\node at (10.5,-0.6) {0};
\node [draw, shape=circle, fill=black] (15) at  (10.5,1.5) {};
\node at (9.9,1.5) {4};
\node [draw, shape=circle, fill=black] (16) at  (10.5,3) {};
\node at (10.5,3.6) {0};
\node [draw, shape=circle, fill=black] (17) at  (12,1.5) {};
\node at (12,2.1) {0};
\node [draw, shape=circle, fill=black] (18) at  (13.5,0) {};
\node at (13.5,-0.6) {0};
\node [draw, shape=circle, fill=black] (19) at  (13.5,1.5) {};
\node at (14.1,1.5) {4};
\node [draw, shape=circle, fill=black] (20) at  (13.5,3) {};
\node at (13.5,3.6) {0};
\node [draw, shape=circle, fill=black] (21) at  (15,0) {};
\node at (15.6,-0.6) {0};
\node [draw, shape=circle, fill=black] (22) at  (15,3) {};
\node at (15.6,3.6) {0};

\draw(1)--(4)--(2);
\draw(5)--(4)--(3);
\draw(4)--(6)--(8);
\draw(7)--(8)--(9);
\draw(10)--(8)--(11);

\draw(12)--(15)--(13);
\draw(16)--(15)--(14);
\draw(15)--(17)--(19);
\draw(18)--(19)--(20);
\draw(21)--(19)--(22);
\end{tikzpicture}
  \caption{A minimum Roman dominating function does not ``generate'' a minimum strong Roman dominating function.}
  \label{fig.Romannominstrfs}
\end{figure}
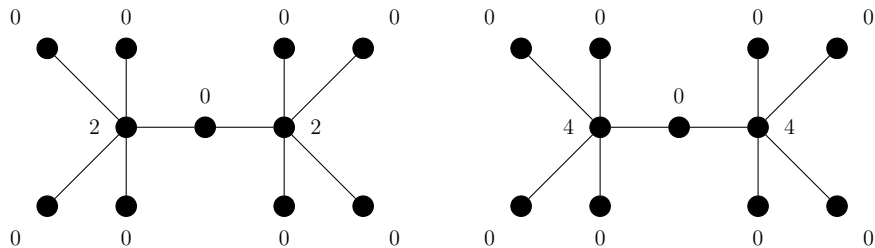

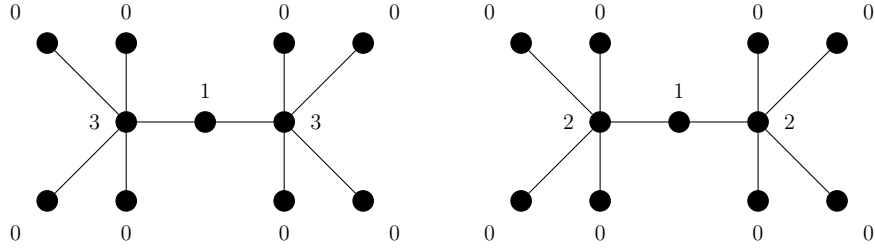
\begin{figure}[h]
\centering
\begin{tikzpicture}[scale=.7, transform shape]
\node [draw, shape=circle, fill=black] (1) at  (0,0) {};
\node at (-0.6,-0.6) {0};
\node [draw, shape=circle, fill=black] (2) at  (0,3) {};
\node at (-0.6,3.6) {0};
\node [draw, shape=circle, fill=black] (3) at  (1.5,0) {};
\node at (1.5,-0.6) {0};
\node [draw, shape=circle, fill=black] (4) at  (1.5,1.5) {};
\node at (0.9,1.5) {3};
\node [draw, shape=circle, fill=black] (5) at  (1.5,3) {};
\node at (1.5,3.6) {0};
\node [draw, shape=circle, fill=black] (6) at  (3,1.5) {};
\node at (3,2.1) {1};
\node [draw, shape=circle, fill=black] (7) at  (4.5,0) {};
\node at (4.5,-0.6) {0};
\node [draw, shape=circle, fill=black] (8) at  (4.5,1.5) {};
\node at (5.1,1.5) {3};
\node [draw, shape=circle, fill=black] (9) at  (4.5,3) {};
\node at (4.5,3.6) {0};
\node [draw, shape=circle, fill=black] (10) at  (6,0) {};
\node at (6.6,-0.6) {0};
\node [draw, shape=circle, fill=black] (11) at  (6,3) {};
\node at (6.6,3.6) {0};

\node [draw, shape=circle, fill=black] (12) at  (9,0) {};
\node at (8.4,-0.6) {0};
\node [draw, shape=circle, fill=black] (13) at  (9,3) {};
\node at (8.4,3.6) {0};
\node [draw, shape=circle, fill=black] (14) at  (10.5,0) {};
\node at (10.5,-0.6) {0};
\node [draw, shape=circle, fill=black] (15) at  (10.5,1.5) {};
\node at (9.9,1.5) {2};
\node [draw, shape=circle, fill=black] (16) at  (10.5,3) {};
\node at (10.5,3.6) {0};
\node [draw, shape=circle, fill=black] (17) at  (12,1.5) {};
\node at (12,2.1) {1};
\node [draw, shape=circle, fill=black] (18) at  (13.5,0) {};
\node at (13.5,-0.6) {0};
\node [draw, shape=circle, fill=black] (19) at  (13.5,1.5) {};
\node at (14.1,1.5) {2};
\node [draw, shape=circle, fill=black] (20) at  (13.5,3) {};
\node at (13.5,3.6) {0};
\node [draw, shape=circle, fill=black] (21) at  (15,0) {};
\node at (15.6,-0.6) {0};
\node [draw, shape=circle, fill=black] (22) at  (15,3) {};
\node at (15.6,3.6) {0};

\draw(1)--(4)--(2);
\draw(5)--(4)--(3);
\draw(4)--(6)--(8);
\draw(7)--(8)--(9);
\draw(10)--(8)--(11);

\draw(12)--(15)--(13);
\draw(16)--(15)--(14);
\draw(15)--(17)--(19);
\draw(18)--(19)--(20);
\draw(21)--(19)--(22);
\end{tikzpicture}
  \caption{A minimum strong Roman dominating function does not ``generate'' a minimum Roman dominating function.}
  \label{fig.strnominromfs}
\end{figure}

We make use of the following results in this section, some of
which are already published or straightforward to observe and so, we omit their proofs.

\begin{prelem}\label{tree}{\em \cite{west}}
For any tree $T$ of order $n\ge 3$, $\gamma_R(T)\le 4n/5$.
\end{prelem}

\begin{prelem}\label{path}{\em\cite{west}}
For paths $P_n$ and cycles $C_n$, $\gamma_R(P_n)=\gamma_R(C_n)=\left\lceil\frac{2n}{3}\right\rceil$.
\end{prelem}

\begin{observation}\label{le2}
For any connected graph $G$ with $\Delta(G)\le 2$, $\gamma_{StR}(G)=\gamma_R(G)$.
\end{observation}

Based on the remark above, from now on, in this work we focus mainly on graphs with maximum degree $\Delta\ge 3$.
\section{The complexity of the strong Roman domination problem}

In this section we deal with the following decision problem. We
must remark that results obtained here are a generalization of
that results previously obtained in \cite{D00}.

$$\begin{tabular}{|l|}
  \hline
  \mbox{STRONG ROMAN DOMINATION PROBLEM}\\
  \mbox{INSTANCE: A non-trivial graph $G$ and a positive integer $r$}\\
  \mbox{PROBLEM: Deciding whether the strong Roman domination number of $G$ is less than $r$}\\
  \hline
\end{tabular}$$

For our purposes of studying the complexity of the STRONG ROMAN
DOMINATION PROBLEM (StRD-Problem for short) we will use the
following variation of the 3-SAT problem which was proved to be
NP-complete in \cite{DPSJvL}. Let $\mathcal{F}$ be a boolean
formula with set of variables $U$ and set of clauses $C$. The
clause-variable graph of $\mathcal{F}$ is defined as the graph
$G_{\mathcal{F}}$  with vertex set $V=U\cup C$ and edge set
$E=\{(v,c)\,:\, v\in V,\, c\in C \mbox{ and } v\in c\}$.

\begin{lemma}{\em \cite{DPSJvL}}\label{lemma-1-neg-3-sat}
The problem of deciding whether a boolean formula $\mathcal{F}$ is
satisfiable is NP-complete, even if
\begin{itemize}
\item every variable occurs exactly once negatively and once or
twice positively, \item every clause contains two or three
distinct variables, \item every clause with three distinct
variables contains at least one negative literal, and \item
$G_{\mathcal{F}}$ is planar.
\end{itemize}
\end{lemma}

In \cite{DPSJvL}, the problem above was called 1-Negative Planar
3-SAT. If there are two equal clauses in a boolean formula
$\mathcal{F}$, then we can reduce such a formula to other one
having all its clauses unique and this does not change the
veracity of the formula. Also, we can consider that the number of
clauses in the boolean formula is greater than or equal to the
number of variables. Therefore, from now on, we will consider a
boolean formula on $n$ variables and $m$ pairwise different
clauses, with $m\ge n$, and satisfying the conditions of Lemma
\ref{lemma-1-neg-3-sat}. Such a boolean formula will be
represented by $\mathcal{F}_3$.

To prove that SRD-Problem is NP-complete, we present a reduction
from 1-Negative Planar 3-SAT. The outline of the procedure behind
this reduction is the following. We begin with an instance of
1-Negative Planar 3-SAT, that is a boolean formula
$\mathcal{F}_3$. We consider a planar embedding of its
clause-variable graph $G_{\mathcal{F}_3}$, and replace each
variable vertex of $G_{\mathcal{F}_3}$ by a variable gadget, and
each clause vertex of $G_{\mathcal{F}_3}$ by a clause gadget.
Hence, we identify the vertices of the variable gadgets and the
vertices of the clause gadgets in its ``corresponding'' way and,
in this sense, we obtain a planar graph $G_{\mathcal{F}_3}$ which
we will use as our instance of SRD-Problem.

We consider the following variable gadgets and clause gadgets. Let
$X = \{a_1, a_2, ..., a_n\}$ (the variables) and $C = \{C_1,
C_2,\ldots, C_m\}$ (the clauses) be an arbitrary instance of
1-Negative Planar 3-SAT with $m\ge n$. The literals are denoted by
$a_i$ (for positive) or $\overline{a_i}$ (for negative). Every
clause $C_i$, is represented by a vertex denoted by $c_i$. Every
variable $a_i$ is represented by a complete bipartite graph
$H_i\cong K_{2,3}$, with partite sets $A_i = \{a_i,
\overline{a_i}\}$ (each one for the corresponding literals of
$a_i$) and $B_i = \{x_i,y_i,z_i\}$. To construct our graph
$G_{F_3}$, we add the edge $a_i\overline{a_i}$, and the two or
three edges connecting each clause vertex $c_i$ with the vertices
corresponding to the literals in the clause $C_i$. In order to be
used while proving our results, since $m\ge n$, we consider a
partition of the vertex set of $G_{F_3}$ into $n$ sets
$S_i=V(H_i)\cup \{c_j\}$, where either $a_i\in C_j$ or
$\overline{a_i}\in C_j$ and a set
$Y=V(G_{F_3})-\left(\bigcup_{i=1}^nS_i\right)$, given by those
vertices $c_l$ not belonging to any $B_i$ (notice that this $Y$
could be empty, in the case $m=n$). Also notice that $|S_i|=6$ for
every $i\in \{1,...,n\}$. We must point out that the same
construction was already used by Paul A. Dreyer \cite{D00} in his
Ph. D. thesis, to study the complexity of the standard Roman
domination, although the procedure of using it was slightly
different.

We will prove that a boolean formula $\mathcal{F}_3$ on $n$
variables and $m$ clauses, being an instance of 1-Negative Planar
3-SAT, has a satisfying truth assignment if and only if the graph
$G_{\mathcal{F}_3}$ satisfies $\gamma_{StR}(G_{\mathcal{F}_3}) =
4n$. Notice that $G_{\mathcal{F}_3}$ has order $m+5n$ and size at
most $3m+7n$. Moreover, every vertex $c_i$ has degree two or three
(since every clause has two or three literals), every vertex $a_i$
has degree five or six, and every $\overline{a_i}$ has degree five
(since each variable appears in $\mathcal{F}_3$ exactly once
negatively and once or twice positively). Figure \ref{example}
shows an example for the case $\mathcal{F}_3=(a_1 \vee a_2
\vee\overline{a_3})\,\wedge\, (\overline{a_1}\vee a_3)\,\wedge\,
(a_1\vee \overline{a_2}\vee a_3)$.  Next we observe some other
properties of the graph $G_{\mathcal{F}_3}$.

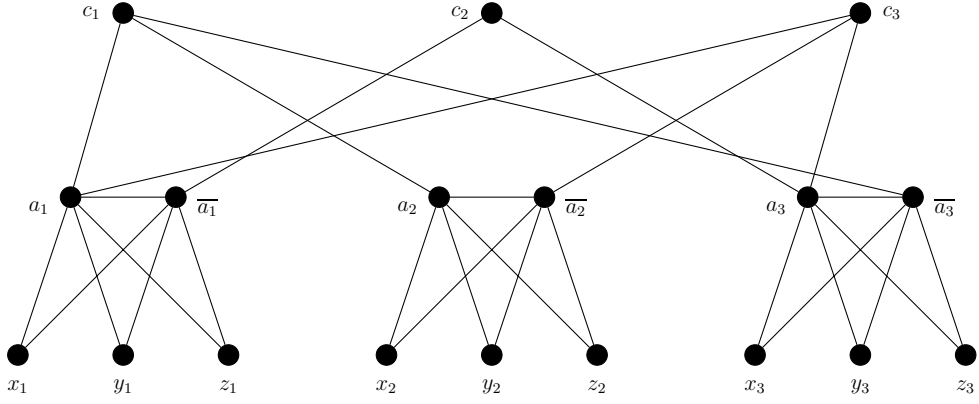
\begin{figure}[h]
\centering
\begin{tikzpicture}[scale=.7, transform shape]
\node [draw, shape=circle, fill=black] (x1) at  (0,0) {}; \node at
(0,-0.6) {$x_1$}; \node [draw, shape=circle, fill=black] (y1) at
(2,0) {}; \node at (2,-0.6) {$y_1$}; \node [draw, shape=circle,
fill=black] (z1) at  (4,0) {}; \node at (4,-0.6) {$z_1$}; \node
[draw, shape=circle, fill=black] (x2) at  (7,0) {}; \node at
(7,-0.6) {$x_2$}; \node [draw, shape=circle, fill=black] (y2) at
(9,0) {}; \node at (9,-0.6) {$y_2$}; \node [draw, shape=circle,
fill=black] (z2) at  (11,0) {}; \node at (11,-0.6) {$z_2$}; \node
[draw, shape=circle, fill=black] (x3) at  (14,0) {}; \node at
(14,-0.6) {$x_3$}; \node [draw, shape=circle, fill=black] (y3) at
(16,0) {}; \node at (16,-0.6) {$y_3$}; \node [draw, shape=circle,
fill=black] (z3) at  (18,0) {}; \node at (18,-0.6) {$z_3$}; \node
[draw, shape=circle, fill=black] (a1) at  (1,3) {}; \node at
(0.4,2.8) {$a_1$}; \node [draw, shape=circle, fill=black] (a11) at
(3,3) {}; \node at (3.6,2.8) {$\overline{a_1}$}; \node [draw,
shape=circle, fill=black] (a2) at  (8,3) {}; \node at (7.4,2.8)
{$a_2$}; \node [draw, shape=circle, fill=black] (a22) at  (10,3)
{}; \node at (10.6,2.8) {$\overline{a_2}$}; \node [draw,
shape=circle, fill=black] (a3) at  (15,3) {}; \node at (14.4,2.8)
{$a_3$}; \node [draw, shape=circle, fill=black] (a33) at  (17,3)
{}; \node at (17.6,2.8) {$\overline{a_3}$}; \node [draw,
shape=circle, fill=black] (c1) at  (2,6.5) {}; \node at (1.4,6.5)
{$c_1$}; \node [draw, shape=circle, fill=black] (c2) at  (9,6.5)
{}; \node at (8.4,6.5) {$c_2$}; \node [draw, shape=circle,
fill=black] (c3) at  (16,6.5) {}; \node at (16.6,6.5) {$c_3$};

\draw(a1)--(x1)--(a11)--(y1)--(a1)--(z1)--(a11)--(a1);
\draw(a2)--(x2)--(a22)--(y2)--(a2)--(z2)--(a22)--(a2);
\draw(a3)--(x3)--(a33)--(y3)--(a3)--(z3)--(a33)--(a3);
\draw(c1)--(a1); \draw(c1)--(a2); \draw(c1)--(a33);
\draw(c2)--(a11); \draw(c2)--(a3); \draw(c3)--(a1);
\draw(c3)--(a22); \draw(c3)--(a3);

\end{tikzpicture}
\caption{The graph $G_{\mathcal{F}_3}$ where $\mathcal{F}_3=(a_1
\vee a_2 \vee\overline{a_3})\,\wedge\, (\overline{a_1}\vee
a_3)\,\wedge\, (a_1\vee \overline{a_2}\vee a_3)$.} \label{example}
\end{figure}

\begin{remark}\label{planar}
$G_{\mathcal{F}_3}$ is planar.
\end{remark}

\begin{proof}
Since all the clauses of $\mathcal{F}_3$ are distinct, not any
copy of the complete bipartite graph $K_{3,3}$ is a subgraph of
$G_{\mathcal{F}_3}$. On the other hand, not any subgraph of
$G_{\mathcal{F}_3}$ is isomorphic to the complete graph $K_5$.
Therefore, by the Kuratowski's theorem is obtained that
$G_{\mathcal{F}_3}$ is planar.
\end{proof}

\begin{remark}\label{mayor-4}
For every subgraph of $G_{\mathcal{F}_3}$ induced by $S_i$, with
$i\in \{1,...,n\}$ and every
$\gamma_{StR}(G_{\mathcal{F}_3})$-function $f$, it follows
$f(S_i)\ge 4$.
\end{remark}

\begin{proof}
Notice that the vertices of the set $B_i\subset S_i$ are pairwise
independent and they have the same shared neighbors (the two
vertices of $A_i$). Let us suppose that $f(S_j)\le 3$ for some
$j\in \{1,...,n\}$. Since $|S_j|\ge 6$, at least three vertices of
$S_j$ have label zero (0) by $f$ and two of them cannot be both in
$A_j$. As a consequence, at least one vertex of $B_j$ has label
zero (0) and at least one vertex $u\in
\{a_j,\overline{a_j}\}=A_j$, has label at least two (2) by $f$. In
this sense, let $t\in \{3,4,5\}$ be the number of vertices in
$S_j$ with label zero (0) by $f$. So, we have that
$f(u)\ge\left\lceil\frac{t}{2}\right\rceil+1$. Therefore, it
follows that
$$f(S_j)\ge f(u)+f(A_j-\{u\})+f(B_j)\ge \left\lceil\frac{t}{2}\right\rceil+1+5-t=6-\left\lfloor\frac{t}{2}\right\rfloor\ge 4,$$
a contradiction. Therefore, $f(S_i)\ge 4$ for every $i\in
\{1,...,n\}$.
\end{proof}

Since, any boolean formula $\mathcal{F}_3$ has $n$ variables, the
result above leads to the following corollary.

\begin{corollary}\label{mayor-4n}
For any boolean formula $\mathcal{F}_3$ on $n$ variables and $m$
clauses with $m\ge n$, being an instance of 1-Negative Planar
3-SAT, $\gamma_{StR}(G_{\mathcal{F}_3})\ge 4n$.
\end{corollary}

Now we present the main result of this section.

\begin{theorem}\label{NP}
Let $\mathcal{F}_3$ be a formula on $n$ variables and $m$ clauses,
being an instance of 1-Negative Planar 3-SAT. Then
$\gamma_{StR}(G_{\mathcal{F}_3})=4n$ if and only if
$\mathcal{F}_3$ is satisfiable.
\end{theorem}

\begin{proof}
We assume that $\mathcal{F}_3$  has satisfying truth assignment.
That is, for any variable $a_i$, we have either $a_i$ or
$\overline{a_i}$ has assigned the value True. We will define a
function $g$ on $V(G_{\mathcal{F}_3})$ in the following way. If
$a_i$ has the value True, then we define $g(a_i) = 4$. On the
contrary, if $a_i$ has the value False, then we define
$g(\overline{a_i}) = 4$. For any other vertex $w\in
V(G_{\mathcal{F}_3})$, we define $g(w) = 0$. Since the definition
of this function is based on the satisfying truth assignment of
$\mathcal{F}_3$, it is straightforward to observe that the
function $g$ is a strong Roman dominating function of weight
$w(g)= 4n$. Thus, $\gamma_{StR}(G_{\mathcal{F}_3})\le 4n$ and, by
Corollary \ref{mayor-4n}, we have that
$\gamma_{StR}(G_{\mathcal{F}_3})=4n$.

On the other hand, we assume that
$\gamma_{StR}(G_{\mathcal{F}_3})=4n$. Since every vertex $a_i$ has
degree five or six, and every $\overline{a_i}$ has degree five
(each variable appears in $\mathcal{F}_3$ exactly once negatively
and once or twice positively), there exists a
$\gamma_{StR}(G_{\mathcal{F}_3})$-function $h$ such that for every
$i\in \{1,...,n\}$, either $h(a_i)=4$ or $h(\overline{a_i})=4$,
and any other vertex of $G_{\mathcal{F}_3}$ has label zero (0) by
$h$. Now, if $h(a_i)=4$, then we set $a_i$ as True. On the
contrary, if $h(\overline{a_i})=4$, then we set $\overline{a_i}$
as False. Since every vertex $c_j$ corresponding to a clause
(notice that it has label zero) is adjacent to at least one vertex
with label four (4), it follows that the clause is satisfied.
Therefore, the formula $\mathcal{F}_3$ has a satisfying truth
assignment.
\end{proof}

As a consequence of Theorem \ref{NP} and Remark \ref{planar} we
have the following result, which completes the proof of the
NP-completeness of the SRD-Problem.

\begin{corollary}
STRONG ROMAN DOMINATION PROBLEM is NP-complete, even when
restricted to planar graphs.
\end{corollary}
\section{Preliminary bounds on the strong Roman domination number}\label{sec.results}

 According to the NP-completeness of the SRD-Problem, it is
therefore desirable to find sharp bounds for the strong Roman
domination number of graphs. In this section, we establish some
sharp bounds for the strong Roman domination number of graphs in terms of
several parameters of the graph.

\begin{proposition}\label{prop-1-delta}
Let $G$ be a graph of order $n$. Then
$$ \gamma_R(G) \le \gamma_{StR}(G) \le \left( 1+\left\lceil\frac{\Delta(G)}{2}\right\rceil
\right)\gamma(G).$$
\end{proposition}

\begin{proof}
Let $f$ be a $\gamma_{StR}(G)$-function.  Define
$\widetilde{f}:V(G)\rightarrow
\{0,1,\ldots,\left\lceil\frac{\Delta(G)}{2}\right\rceil+1\}$ by
$\widetilde{f}(v)=2$ for $v\in B_2$ and $\widetilde{f}(v)=f(v)$
otherwise. It is straightforward to observe that $\widetilde{f}$
is a Roman dominating function of $G$ and hence $\gamma_R(G) \le
w(\widetilde{f})\le w(f)= \gamma_{StR}(G)$

To prove the upper bound, let $D$ be a dominating set of minimum
cardinality and define $h:V(G)\rightarrow
\{0,1,\ldots,\left\lceil\frac{\Delta(G)}{2}\right\rceil+1\}$ by
$h(v)=1+ \left\lceil \frac{\Delta(G)}{2}\right\rceil$ for $v\in D$
and $h(v)=0$ otherwise. Obviously,  $h$ is a strong
Roman dominating function for $G$ and, as a consequence,
$\gamma_{StR}(G) \le w(h)= \left (1+ \left\lceil
\frac{\Delta(G)}{2} \right\rceil\right ) \gamma(G)$.
\end{proof}

\begin{proposition}\label{prop-2-delta}
Let $G$ be a graph of order $n$. Then
$$ \gamma_{StR}(G) \le n - \left\lfloor \frac
{\Delta(G)}{2}\right\rfloor.$$
\end{proposition}

\begin{proof}
Let $v \in V(G)$ be a vertex of maximum degree $\Delta(G)$ and
define $f:V(G)\rightarrow
\{0,1,\ldots,\left\lceil\frac{\Delta(G)}{2}\right\rceil+1\}$ by
$f(v)=1+ \left\lceil \frac{\Delta(G)}{2} \right\rceil$, $f(x)=0$
for $x\in N(v)$ and $f(x)=1$ otherwise.  Clearly, $f$ is a strong Roman dominating function for $G$ and so $\gamma_{StR}(G) \le w(f)=1+
\left\lceil \frac{\Delta}{2} \right\rceil+n-\Delta-1= n - \left\lfloor \frac {\Delta}{2}\right\rfloor$. This completes the proof.
\end{proof}
 An immediate consequence of Observation \ref{le2} and
Proposition \ref{prop-2-delta} now follows.

\begin{corollary}\label{value-n}
Let $G$ be a connected graph or order $n$. Then $\gamma_{StR}(G)=n$ if and only if $G=K_1\;{\rm or}\;K_2$.
\end{corollary}

The two propositions above give two different upper bounds on
$\gamma_{StR}(G)$ which are involving the maximum degree $\Delta$
of the graph. So, an interesting question regarding this could be:
Can we compare them with respect to its efficiency? As we can
observe in the next two examples, such a question could be not
completely clearly addressed.

Notice that the left hand side graph of Figure
\ref{fig.cotamin_2fs} satisfies that  $\left(
1+\left\lceil\frac{\Delta}{2}\right\rceil \right)\gamma(G)=6<7=n -
\left\lfloor \frac {\Delta}{2}\right\rfloor$. Nevertheless, the
right hand side graph shown in Figure \ref{fig.cotamin_2fs}
carries out that $n - \left\lfloor \frac
{\Delta}{2}\right\rfloor=4<6=\left(
1+\left\lceil\frac{\Delta}{2}\right\rceil \right)\gamma(G)$.

\begin{figure}[h]
\begin{center}
\begin{tabbing}
$\;\;\;\;\;\;\;\;\;\;$\=$\;\;\;\;\;\;\;\;\;\;$\=$\;\;\;\;\;\;\;\;\;\;$
\=
\begin{tikzpicture}[scale=.7, transform shape]
\node [draw, shape=circle, fill=black] (1) at  (0,0) {}; \node at
(-0.6,0) {0}; \node [draw, shape=circle, fill=black] (2) at
(0,1.5) {}; \node at (-0.6,1.5) {0}; \node [draw, shape=circle,
fill=black] (3) at  (0,3) {}; \node at (-0.6,3) {0}; \node [draw,
shape=circle, fill=black] (4) at  (1.5,1.5) {}; \node at (1.5,2.1)
{3}; \node [draw, shape=circle, fill=black] (5) at  (3,1.5) {};
\node at (3,2.1) {0}; \node [draw, shape=circle, fill=black] (6)
at  (4.5,1.5) {}; \node at (4.5,2.1) {3}; \node [draw,
shape=circle, fill=black] (7) at  (6,0) {}; \node at (6.6,0) {0};
\node [draw, shape=circle, fill=black] (8) at  (6,1.5) {}; \node
at (6.6,1.5) {0}; \node [draw, shape=circle, fill=black] (9) at
(6,3) {}; \node at (6.6,3) {0};

\draw(1)--(4)--(3); \draw(2)--(4)--(5)--(6)--(8);
\draw(9)--(6)--(7);
\end{tikzpicture}
\=$\;\;\;\;\;\;\;\;\;\;$\=
\begin{tikzpicture}[scale=.7, transform shape]
\node [draw, shape=circle, fill=black] (1) at  (0,0) {}; \node at
(-0.6,0) {0}; \node [draw, shape=circle, fill=black] (3) at  (0,3)
{}; \node at (-0.6,3) {0}; \node [draw, shape=circle, fill=black]
(4) at  (1.5,1.5) {}; \node at (1.5,2.1) {3}; \node [draw,
shape=circle, fill=black] (5) at  (3,1.5) {}; \node at (3,2.1)
{0}; \node [draw, shape=circle, fill=black] (6) at  (4.5,1.5) {};
\node at (4.5,2.1) {1};

\draw(1)--(4)--(3); \draw(4)--(5)--(6);
\end{tikzpicture}
\end{tabbing}
\caption{Graph used to compare upper bound of Propositions
\ref{prop-1-delta} and \ref{prop-2-delta}.}\label{fig.cotamin_2fs}
\end{center}
\end{figure}
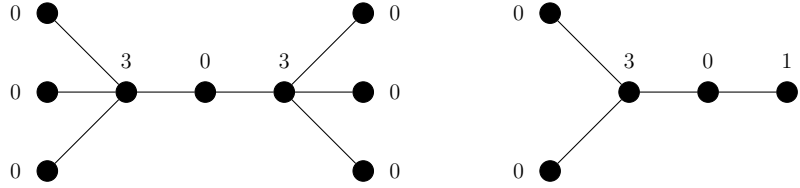

A \emph{rooted graph} is a graph in which one vertex is labeled
in a special way so as to distinguish it from other vertices. The
special vertex is called the \emph{root} of the graph. Let $G$ be
a labeled graph on $n$ vertices. Let ${\cal H}$ be an ordered
sequence of $n$ rooted graphs $H_1$, $H_2$,...,$H_n$. The
\emph{rooted product graph} $G( {\cal H})$ is the graph obtained
by identifying the root of $H_i$ with the $i^{th}$ vertex of $G$
\cite{rooted-first}. If the family ${\cal H}$ consists of $n$
isomorphic rooted graphs, then we use the notation  $G\circ_{v}
H$. Moreover, if $H$ is a vertex transitive graph, then $G\circ_v
H$ does not depend on the choice of $v$, up to isomorphism. In
such a case  we will just write $G\circ H$.


\begin{proposition}\label{diam}
For any connected graph $G$ on $n$ vertices,
$$\gamma_{StR}(G)\le n-\left\lfloor\frac{1+{\rm diam}(G)}{3}\right\rfloor.$$
Furthermore, this bound is sharp
for the rooted product graph $P_{3}\circ P_2$.
\end{proposition}

\begin{proof}
Let $P=v_1v_2\ldots v_{q}$, $q={\rm diam}(G)+1$, be a diametral path in $G$ and let
$f$ be a $\gamma_{StR}(P)$-function. Define $g:V(G)\rightarrow
\{0,1,\ldots,\left\lceil\frac{\Delta(G)}{2}\right\rceil+1\}$ by $g(x)=f(x)$ for
$x\in V(P)$ and $g(x)=1$ otherwise. Obviously $g$ is a strong Roman dominating
function of $G$. Hence, $$\gamma_{StR}(G)\le \omega(f)+(n-{\rm diam}(G)-1)= n-\left\lfloor\frac{1+{\rm diam}(G)}{3}\right\rfloor.$$
\end{proof}

\begin{proposition}\label{girth}
Let $G$ be a connected graph of order $n$ with $g(G)\ge 3$. Then
$$\gamma_{StR}(G)\le n-\left\lfloor\frac{g(G)}{3}\right\rfloor.$$
Furthermore, the bound is sharp
for the rooted product graph $C_{3}\circ P_2$.
\end{proposition}

\begin{proof}
Assume $C$ is a cycle of $G$ with $g(G)$ edges. Let $f$ be a
$\gamma_{StR}(C)$-function and define $g:V(G)\rightarrow
\{0,1,\ldots,\left\lceil\frac{\Delta(G)}{2}\right\rceil+1\}$ by
$g(x)=f(x)$ for $x\in V(C)$ and $g(x)=1$ otherwise. Clearly $g$ is
a strong Roman dominating function of $G$ that implies
$\gamma_{StR}(G)\le
n-|V(C)|+\left\lceil\frac{2g(G)}{3}\right\rceil=n-\left\lfloor\frac{g(G)}{3}\right\rfloor.$
\end{proof}

Next we continue with another upper bound on the strong Roman
domination number of graphs.  The proof of such
a bound uses a probabilistic approach, which in some
sense, is a generalization
of a similar result presented in \cite{CDHH04}.

\begin{proposition}\label{prop-probab}
Let $G$ be a graph of order $n$, minimum degree $\delta$ and
maximum degree $\Delta$, such that
$\left\lceil\frac{\Delta}{2}\right\rceil < \delta$. Then
$$\gamma_{StR}(G) \le \frac{\left( 1+\left\lceil\frac{\Delta}{2}\right\rceil \right) n }{\delta+1}
            \left( \ln \left( \frac{1+\delta}{1+\left\lceil\frac{\Delta}{2}\right\rceil} \right) +1  \right).$$
\end{proposition}

\begin{proof}
Given a set $A\subset V(G)$, we consider the set $B=V(G)-N[A]$. We
denote by $A^c=\{v\in V(G)\,:\,v\notin A\}$. Notice that
$B=A^c\cap N(A)^c=N[A]^c$. Given $v\in V(G)$, we define $p$ as the
probability that $v$ would belong to $A$, $p=P[v \in A]$. We
consider now $P[v \in B]$ as the probability of the event such
that $v$ does not belong to $A$ and also, the neighbors of $v$ are
not in $A$. That is, $P[v \in B]=P[v \in A^{c} \cap
N(A)^{c}]=(1-p)(1-p)^{\delta(v)}=(1-p)^{1+\delta(v)} \le
(1-p)^{1+\delta}$. According to this, we can approximate the
expected values of $|A|$ and $|B|$: $E[|A|]=np$ and
$E[|B|]=n(1-p)^{1+\delta(v)} \le n(1-p)^{1+\delta(G)} \le n
e^{-p(1+\delta)}$. Now, define the function $f:V(G)\rightarrow
\{0,1,\ldots,1+\left\lceil\frac{\Delta}{2}\right\rceil\}$ by
$$f(v)=\left\{\begin{array}{ll}
                            1+\left\lceil\frac{\Delta}{2}\right\rceil & \mbox{if $v \in A$} \\
                            0 & \mbox{if $v \in N(A)$} \\
                            1 & \mbox{if $v \in B$.}
                          \end{array}
\right.$$ Then, the expected value of $f(V)$ is:
\begin{align*}
E[f(V)]&=\left (1+\left\lceil\frac{\Delta}{2}\right\rceil \right)E[|A|]+E[|B|]\\
&\le\left (1+\left\lceil\frac{\Delta}{2}\right\rceil
\right)np+ne^{-p(1+\delta)}.
\end{align*}
The last expression attains its minimum value $\left (1+
\left\lceil\frac{\Delta}{2}\right\rceil
\right)n-n(1+\delta)e^{-p(1+\delta)}$ if and only if
$e^{-p(1+\delta)}=\frac{1+ \left\lceil\frac{\Delta}{2}\right\rceil
}{\delta+1}$. Moreover, the last equality has solution $p$ such
that $0<p<1$, if $\frac{1+ \left\lceil\frac{\Delta}{2}\right\rceil
}{\delta+1}<1$, which means $\frac{1+ \delta }{1+
\left\lceil\frac{\Delta}{2}\right\rceil}>1$. This leads to
$\left\lceil\frac{\Delta}{2}\right\rceil < \delta$ and, as a
consequence, the solution is $p=\frac{1}{1+ \delta}\ln\left(
\frac{1+ \delta }{1+
\left\lceil\frac{\Delta}{2}\right\rceil}\right)$. Therefore,
$\gamma_{StR}(G) \le \left(1+
\left\lceil\frac{\Delta}{2}\right\rceil \right)\frac{n}{1+ \delta
}\ln\left(\frac{1+ \delta }{1+
\left\lceil\frac{\Delta}{2}\right\rceil} \right)+\left(1+
\left\lceil\frac{\Delta}{2}\right\rceil \right)\frac{n}{1+
\delta}$, which leads to the result.
\end{proof}

 The last proposition provides a new upper bound for
$\gamma_{StR}(G)$, which we can compare with some of the previous
upper bounds. For instance, we could compare it with the one of
Proposition \ref{prop-2-delta}. That is, we want to find those
graphs $G$ with \begin{equation}\label{pp}\frac{\left(
1+\left\lceil\frac{\Delta}{2}\right\rceil \right) n }{\delta+1}
\left( \ln \left(
\frac{1+\delta}{1+\left\lceil\frac{\Delta}{2}\right\rceil} \right)
+1 \right)\le n - \left\lfloor \frac
{\Delta}{2}\right\rfloor.\end{equation} As an example, we consider
those graphs $G$ with above property such that
$\left\lceil\frac{\Delta}{2}\right\rceil = \delta-1$. Some
algebraic work on (\ref{pp}) leads to a graph $G$ satisfying
$\delta\le \sqrt{n+1}$. For such a graph, the bound of Proposition
\ref{prop-probab} is better than the one in Proposition
\ref{prop-2-delta}. Otherwise, vise versa. Similar conclusions
could be extracted using other different statements, which makes
that the process of comparing all these bounds above is not
clear.

We observe that a strong Roman dominating defensive strategy
needs, in general, more legions than a Roman dominating one, so
the advantage is not to safe resources but to design an stronger
empire against external attacks. Under the strong Roman dominating
strategy, any strong vertex must be able to defend itself and at
least one half of its weak neighbors. The goal is then to deal
with this situation by using as few resources (legions) as
possible.

Next result gives the minimum number of legions which are needed
to protect the Roman empire fortifications under the strong Roman
domination strategy.

\begin{proposition}\label{gminimo}
Let $G$ a graph of order $n$. Then $$\gamma_{StR}(G)\ge
\left\lceil \frac{n+1}{2} \right\rceil.$$
 Moreover, if $n$ is odd, then equality holds if
and only if $\Delta(G)=n-1$.
\end{proposition}

\begin{proof}
Let $f$ be a $\gamma_{StR}(G)$-function and let $B_{1}=\{ w \in
V(G)\mid f(w)=1\}$, $B_{2}=\{w \in V(G) \mid f(w)\geq 2\}$,
$B_{0}=\{w \in V(G) \mid f(w)= 0\}$, $B_{0}^{1}=\{w \in B_{0} \mid
|N(w)\cap B_{2}|=1 \}$ and $B_{0}^{2}=B_0\setminus B_{0}^{1}$.
 Clearly $(B_0, B_1, B_2)$ is a partition of $V(G)$ and
$(B_0^1, B_0^2)$ is a partition of $B_0$. Hence,
$n=|B_{1}|+|B_{2}|+|B_{0}^{1}|+|B_{0}^{2}|=|B_{1}|+|B_{2}|+|B_{0}|$.
It follows that

\begin{align*}
\gamma_{StR}(G)&=\sum_{v\in B_1}f(v)+\sum_{v\in B_2}f(v)\\
&\ge |B_{1}|+|B_{2}| + \frac{1}{2}|B_{0}^{1}| +  \sum_{w \in B_{0}^{2}}{\frac{1}{2} |N(w)\cap B_{2}|}\\
&\ge|B_{1}|+|B_{2}|+\frac{1}{2}|B_{0}^{1}|+|B_{0}^{2}|\\
&=n-\frac{1}{2}|B_{0}^{1}|\\
&\ge n-\frac{n-1}{2}\quad\;\;\mbox{(since $|B_{0}^{1}| \le n-1$)}\\
&=\frac{n+1}{2}.
\end{align*}
Therefore, the result follows, since $\gamma_{StR}(G)$ is an
integer number.

 If $n$ is odd and $\Delta(G)=n-1$, then we deduce from Proposition \ref{prop-2-delta} that
$\gamma_{StR}(G)\le
n-\left\lfloor\frac{\Delta(G)}{2}\right\rfloor=\frac{n+1}{2}$ and so
$\gamma_{StR}(G)=\frac{n+1}{2}$.

Conversely, let $n$ is odd and $\gamma_{StR}(G)=\frac{n+1}{2}$.
Then the inequalities occurring in the proof become equalities.
Hence $|B_0^1|=n-1$ and $|B_2|=1$ that implies $\Delta(G)=n-1$ and
the proof is completed.
\end{proof}

Next result is an immediate consequence of Propositions \ref{prop-2-delta} and \ref{gminimo}.

\begin{corollary}\label{star}
For $n\ge 1$, $\gamma_{StR}(K_{1,n})=\left\lceil\frac{n+1}{2}\right\rceil$.
\end{corollary}

\section{Trees}

 In this section we first show that for any tree $T$ of order $n\ge
3$, $\gamma_{StR}(T)\le \frac{6n}{7}$ and then, we characterize all extremal trees
which attain this upper bound.
To begin with, we need to introduce some terminology
and notation. A vertex of degree one is called a \emph{leaf}, and
its neighbor is called a {\em stem}. If $v$ is an stem, then $L_v$
will denote the set of all leaves adjacent to $v$. An stem $v$ is
called {\em end-stem} if $|L_v|>1$. For $r,s\ge 1$, a \emph{double star}
$S(r,s)$ is a tree with exactly two vertices that are not leaves,
with one adjacent to $r$ leaves and the other to $s$ leaves. For a
vertex $v$ in a rooted tree $T$, let $C(v)$ denotes the set of
children of $v$, $D(v)$ denotes the set of descendants of $v$ and
$D[v]=D(v)\cup \{v\}$. Also, the \emph{depth} of $v$, ${\rm depth}(v)$,
is the largest distance from $v$ to a vertex in $D(v)$. The {\em
maximal subtree} at $v$ is the subtree of $T$ induced by $D(v)\cup
\{v\}$, and is denoted by $T_v$.

A {\em subdivision} of an edge $uv$ is obtained by replacing the
edge $uv$ with a path $uwv$, where $w$ is a new vertex. The {\em
subdivision graph} $S(G)$ is the graph obtained from $G$ by
subdividing each edge of $G$. The  subdivision star $S(K_{1,t})$
for $t\ge 2$, is called a {\em healthy spider} $S_{t,t}$. A {\em
wounded spider} $S_{t,q}$ is the graph formed by subdividing $q$
of the edges of a star $K_{1,t}$ for $t\ge 2$ where $q\le t-1$.
Note that stars are wounded spiders. A {\em spider}  is a healthy
or wounded spider. We now present a result on the strong Roman domination of double stars. Notice that spiders and double stars can be also represented as rooted product graphs.

\begin{lemma}\label{dstar}
For any integers $1\le p\le q$ and any double star $T=S(p,q)$ of order $n=p+q+2$, $\gamma_{StR}(T)<\frac{6n}{7}$.
\end{lemma}

\begin{proof}
Let $u,v$ be the non-central vertices of $T$. If $q=1$, then also $p=1$. So,
$T=P_4$ and we have $\gamma_{StR}(P_4)=3< \frac{6n}{7}$. Assume
that $q\ge 2$. First let $p=1$. Define $f$ on $V(T)$ by assigning
$1+\left\lceil\frac{q+1}{2}\right\rceil$ to $v$, $1$ to the leaf at distance
$2$ from $v$ and $0$ to the other vertices. Obviously $f$ is a
StRDF of $T$ of weight $\left\lceil\frac{n+2}{2}\right\rceil$ and we have
$\gamma_{StR}(T)=\left\lceil\frac{n+2}{2}\right\rceil< \frac{6n}{7}$ because
$n\ge 5$. Now let $p\ge 2$. Define $f$ on $V(T)$ by assigning
$1+\left\lceil\frac{q}{2}\right\rceil$ to $v$, $1+\left\lceil\frac{p}{2}\right\rceil$ to
$u$, and 0 to the remaining vertices. Obviously $f$ is a StRDF of
$T$ of weight $2+\left\lceil\frac{p}{2}\right\rceil+\left\lceil\frac{q}{2}\right\rceil$.
Considering the parity of $p$ and $q$, it is straightforward to
see that
$\gamma_{StR}(T)=2+\left\lceil\frac{p}{2}\right\rceil+\left\lceil\frac{q}{2}\right\rceil<\frac{6n}{7}$
and the proof is complete.
\end{proof}

The following bound is the main result of this section, where we bound the strong Roman domination of trees in terms of its order.

\begin{theorem}\label{tree1}
If $T$ is a tree of order $n\ge 3$, then $$\gamma_{StR}(T)\le \frac{6n}{7}.$$
\end{theorem}

\begin{proof}
We proceed by induction on $n\ge 3$. The statement holds for all
trees of order $n\le 5$. For the inductive hypothesis, let $n\ge
6$ and suppose that for every tree $T$  of order at least 3 and
less than $n$ the result is true. Let $T$ be a tree of order $n\ge
6$. If ${\rm diam}(T)=2$, then $T$ is a star, which yields
$\gamma_{StR}(T)=\left\lceil\frac{n+1}{2}\right\rceil<(6n)/7$ by
Proposition \ref{star}. If ${\rm diam}(T)=3$, then $T$ is a double
star and the result follows from Lemma \ref{dstar}. If $T$ is a
path, then by Observation \ref{le2} and Theorem \ref{tree} we have
$\gamma_{StR}(T)=\gamma_{R}(T)\le \frac{4n}{5}<\frac{6n}{7}$.
Thus, we may assume that ${\rm diam}(T)\ge 4$ and $\Delta(T)\ge
3$. For a subtree $T'$ with $n'$ vertices, where $n'\ge 3$, the
induction hypothesis yields a StRDF $f'$ of $T'$ with weight at
most $\frac{6n'}{7}$. We shall find a subtree $T'$ such that
adding a bit more weight to $f'$ will yield a small enough StRDF
$f$ for $T$. Let $P=v_1v_2\ldots v_k$ be a diametral path in $T$
chosen to maximize $t=\deg_T(v_2)$. Also suppose among  paths with
this property we choose a path such that $|L_{v_3}|$ is as large
as possible. Root $T$ at $v_k$. We consider three cases.

\smallskip
\noindent {\bf Case 1:}\quad $t\ge 4$.\\ Let $T'=T-T_{v_2}$. Since
${\rm diam}(T)\ge 4$, we have $n'\ge 3$. Define $f$ on $V(T)$ by
letting $f(x)=f'(x)$ except for
$f(v_{2})=1+\left\lceil\frac{t}{2}\right\rceil$ and $f(x)=0$ for each $x\in
L_{v_2}$. Note that $f$ is a StRDF for $T$ and that
$$w(f)=w(f')+1+\left\lceil\frac{t}{2}\right\rceil\le \frac{6(n-t)}{7}+1+\left\lceil\frac{t}{2}\right\rceil< \frac{6n}{7}.$$

By the choice of the diametral path, it remains only to consider those trees, where all end-stems on diametral paths
have degree at most 3.

\smallskip
\noindent {\bf Case 2:}\quad $t=3$.\\ Let $L_{v_2}=\{v_1,u\}$. We
consider the following subcases regarding the degree of the non leaf neighbor of $v_2$ in the diametral path $P$.

\smallskip
{\bf Subcase 2.1.} $\deg_T(v_3)=2$.\\ Let $T'=T-T_{v_3}$. If
$|V(T')|=2$, then $|V(T)|=6$ and it is easy to see that
$\gamma_{StR}(T)=5<\frac{6n}{7}$. Suppose that $|V(T')|\ge 3$.
Define $f$ on $V(T)$ by $f(x)=f'(x)$ for every $x\in V(T')$, $f(v_3)=1,
f(v_{2})=2$ and $f(x)=0$ for $x\in L_{v_2}$. Clearly $f$ is a
StRDF for $T$ of weight $w(f')+3$ and so,
$$w(f)=w(f')+3\le \frac{6(n-4)}{7}+3< \frac{6n}{7}.$$

\smallskip
{\bf Subcase 2.2.} $\deg_T(v_3)\ge 3$ and $v_3$ is adjacent to an end-stem $w$ of degree 3 such that $w\neq v_2$.\\
Assume that $L_w=\{w_1,w_2\}$ and let
$T'=T-\{v_2,v_1,u,w,w_1,w_2\}$. If $|V(T')|=2$, then $|V(T)|=8$
and it is easy to see that $\gamma_{StR}(T)=6<\frac{6n}{7}$.
Suppose $|V(T')|\ge 3$. Define $f$ on $V(T)$ by $f(x)=f'(x)$ for every
$x\in V(T')-\{v_3\}, f(v_3)=f'(v_3)+1$, $f(v_2)=f(w)=2$ and
$f(x)=0$ for $x\in L_w\cup L_{v_3}$. Obviously $f$ is a StRDF for
$T$ of weight $w(f')+5$, and so
$$w(f)=w(f')+5\le \frac{6(n-6)}{7}+5< \frac{6n}{7}.$$

By Subcase 2.2, we only need to consider the possibilities in which all end-stems adjacent to
$v_3$, with exception $v_2$, have degree 2. If ${\rm diam}(T)=4$,
then $T-T_{v_2}$ is a spider, and it is easy to see that
$\gamma_{StR}(T)\le \frac{6n}{7}$. Let ${\rm diam}(T)\ge 5$

\smallskip
{\bf Subcase 2.3.} $\deg_T(v_3)\ge 4$ is an even number.\\ Let
$T'=T-D(v_3)$. Hence, $v_3$ is a leaf in $T'$ and also
$|f'(v_3)|\ge 1$ or $|f'(v_4)|\ge 1$. Define $f$ on $V(T)$ by
assigning $f(x)=f'(x)$ for every $x\in V(T')-\{v_3\}$,
$f'(v_3)+\left\lfloor\frac{\deg(v_3)}{2}\right\rfloor$ to $v_3$, 2 to
$v_2$, 1 to the leaf adjacent to an end-stem of degree 2 in
$T_{v_3}$, and 0 to the remaining vertices. Notice that $f$ is a
StRDF for $T$ of weight
$w(f')+\left\lfloor\frac{\deg(v_3)}{2}\right\rfloor+2+r$, where $r$ is the
number of end-stems of degree 2 in $T_{v_3}$. Since
$|D(v_3)|=\deg(v_3)+r+1$, we deduce that $$w(f)\le
\frac{6(n-|D(v_3)|)}{7}+\left\lfloor\frac{\deg(v_3)}{2}\right\rfloor+2+r<
\frac{6n}{7}.$$

\smallskip
{\bf Subcase 2.4.} $\deg_T(v_3)\ge 5$ is odd.\\ Let
$T'=T-T_{v_3}$. Define $f$ on $V(T)$ by assigning $f(x)=f'(x)$ for every
$x\in V(T')$, $\left\lceil\frac{\deg(v_3)}{2}\right\rceil$ to $v_3$, 2 to
$v_2$, 1 to the leaf adjacent to an end-stem of degree 2 in
$T_{v_3}$, and 0 to the remaining vertices. Observe that $f$ is a
StRDF for $T$ of weight
$w(f')+\left\lceil\frac{\deg(v_3)}{2}\right\rceil+2+r$, where $r$ is the
number of end-stems of degree 2 in $T_{v_3}$. Since
$n=|V(T')|+\deg(v_3)+r+2$, we deduce that $$w(f)\le
\frac{6(n-\deg(v_3)-r-2)}{7}+\left\lceil\frac{\deg(v_3)}{2}\right\rceil+2+r<
\frac{6n}{7}.$$

\smallskip
{\bf Subcase 2.5.} $\deg(v_3)=3$ and $v_3$ is adjacent to an end-stem $w$ of degree 2.\\
Suppose $w'$ is the leaf adjacent to $w$. Let $T'=T-T_{v_3}$.
Define $f$ on $V(T)$ by $f(x)=f'(x)$ for every $x\in V(T'),
f(v_3)=f(v_2)=2, f(w')=1$, and $f(v_1)=f(u)=0$. Clearly, $f$ is a
StRDF for $T$ of weight $w(f')+5$ and so,
$$w(f)=w(f')+5\le \frac{6(n-6)}{7}+5< \frac{6n}{7}.$$

\smallskip
{\bf Subcase 2.6.} $\deg(v_3)=3$ and $v_3$ is adjacent to a leaf $w$.\\
Suppose $T'=T-T_{v_3}$ and define $f$ on $V(T)$ by $f(x)=f'(x)$
for every $x\in V(T'), f(v_3)=f(v_2)=2$, and $f(v_1)=f(u)=f(w)=0$.
Obviously, $f$ is a StRDF for $T$ of weight $w(f')+4$ and
as above
$$w(f)=w(f')+4\le \frac{6(n-5)}{7}+4< \frac{6n}{7}.$$

\smallskip
\noindent {\bf Case 3:}\quad $t=2$.\\ By the choice of the diametral
path, we only need to consider those possibilities in which every end-stem on a diametral path has
degree 2. In particular, when every end-stem adjacent to $v_3$ has
degree 2. Thus, it follows that $T_{v_3}$ is a spider. Assume that
$\delta_T(v_3)=d$ and $r$ is the number of end-stems adjacent to
$v_3$ in $T_{v_3}$.

First let ${\rm diam}(T)=4$. Hence, $T$ is a spider obtained from a
star $K_{1,d}$ by subdividing $r+1$ edges where $2\le r+1\le d$.
So $|V(T)|=n=d+r+2$. Define $f:V(T)\rightarrow
\{0,\ldots,1+\left\lceil\frac{d}{2}\right\rceil\}$ by
$f(v_{3})=1+\left\lceil\frac{d}{2}\right\rceil, f(x)=0$ for every $x\in N(v_3)$ and
$f(x)=1$ otherwise. Clearly, $f$ is a StRDF on $G$ of weight
$r+\left\lceil\frac{d}{2}\right\rceil+2$. If $d$ is even, then it is easy to
see that $\gamma_{StR}(T)\le
r+\left\lceil\frac{d}{2}\right\rceil+2<\frac{6n}{7}$.  If $d$ is odd and
$r+1\le d-1$, then we obtain $\gamma_{StR}(T)\le
r+\frac{d+1}{2}+2<\frac{6n}{7}$. Assume that $d$ is odd and
$d=r+1$. Thus, $\gamma_{StR}(T)\le r+\frac{d+1}{2}+2\le
\frac{6n}{7}$ with equality if and only if $d=r+1=3$ and this happen if
and only if $T=S(K_{1,3})$.

Now let ${\rm diam}(T)\ge 5$. Consider the following subcases.

\smallskip
{\bf Subcase 3.1.} $\deg_T(v_3)\ge 3$. We distinguish some possibilities.
\begin{description}
    \item[(a)] $\deg_T(v_3)\ge 4$ is even.\\
Let $T'=T-T_{v_3}$. Define $f:V(T)\rightarrow
\{0,\ldots,1+\left\lceil\frac{\Delta(G)}{2}\right\rceil\}$ by assigning
$f(x)=f'(x)$ for every $x\in V(T')$, $1+\frac{d}{2}$ to $v_3$, 1 to
leaves at distance 2 from $v_3$ and 0 to the vertices in
$N(v_3)-\{v_4\}$. We notice that $f$ is an StRDF of $T$ of
weight $\omega(f')+1+\frac{d}{2}+r$. By the induction hypothesis
we obtain $$w(f)\le \frac{6(n-d-r)}{7}+1+\frac{d}{2}+r\le
\frac{6n}{7},$$ with equality if and only if $d=4$ and $r=3$.
    \item[(b)] $\deg_T(v_3)\ge 7$ is odd.\\
Let $T'=T-T_{v_3}$. Define $f:V(T)\rightarrow
\{0,\ldots,1+\left\lceil\frac{\Delta(G)}{2}\right\rceil\}$ by assigning
$f(x)=f'(x)$ for every $x\in V(T')$, $1+\frac{d+1}{2}$ to $v_3$, 1 to
leaves at distance 2 from $v_3$ and 0 to the vertices in
$N(v_3)-\{v_4\}$. Obviously $f$ is an StRDF of $T$ of weight
$\omega(f')+1+\frac{d+1}{2}+r$, and it follows from the induction
hypothesis that $$w(f)\le \frac{6(n-d-r)}{7}+1+\frac{d+1}{2}+r<
\frac{6n}{7}.$$
    \item[(c)] $\deg_T(v_3)=5$.\\ \begin{itemize}
        \item $r=4$. Assume that $N(v_3)\setminus
    \{v_2,v_4\}=\{w_1,w_2,w_3\}$ and let $w_i'$ be the leaf adjacent to $w_i$ for each
    $i=1,2,3$. Let $T'=T-\{v_1,v_2,w_1,w_1',w_2',w_3'\}$. Hence,
    either $f'(v_4)=0$ and $|f'(v_3)|+|f'(w_2)|+|f'(w_3)|\ge 3$, or $f'(v_4)\neq0$ and $|f'(v_3)|+|f'(w_2)|+|f'(w_3)|\ge
    2$. Define $f:V(T)\rightarrow
\{0,\ldots,1+\left\lceil\frac{\Delta(G)}{2}\right\rceil\}$ by $f(x)=f'(x)$
for every $x\in V(T')-\{v_3,w_2,w_3\}$,
$f(v_3)=|f'(v_3)|+|f'(w_2)|+|f'(w_3)|+1$, $f(x)=0$ for every $x\in
N(v_3)-\{v_4\}$ and $f(x)=1$ otherwise. It is easy to see that $f$
is an StRDF of $T$ of weight $\omega(f')+5$ and, by the induction
hypothesis, we have $w(f)\le \frac{6(n-6)}{7}+5< \frac{6n}{7}.$
        \item $r=3$. Let $L_{v_3}=\{w_3\}$, $N(v_3)\setminus
    (\{v_2,v_4\}\cup L_{v_3})=\{w_1,w_2\}$ and $w_i'$ be the leaf adjacent to $w_i$ for each
    $i=1,2$. Suppose that $T'=T-\{v_1,v_2,w_1,w_1',w_2'\}$. Hence,
    either $f'(v_4)=0$ and $|f'(v_3)|+|f'(w_3)|+|f'(w_3)|\ge 3$, or $f'(v_4)\neq0$ and $|f'(v_3)|+|f'(w_2)|+|f'(w_3)|\ge
    2$. The function $f$ defined as [(a)], is clearly an StRDF of $T$ of weight $\omega(f')+4$ and, as above we have
    $w(f)\le \frac{6(n-5)}{7}+4< \frac{6n}{7}.$
        \item $r=2$. Suppose that $L_{v_3}=\{w_2,w_3\}$ and $N(v_3)\setminus
    (\{v_2,v_4\}\cup L_{v_3})=\{w_1\}$ and $w_1'$ is the leaf adjacent to
    $w_1$. Let $T'=T-\{v_1,v_2,w_1,w_1'\}$. Using an argument
    similar to that described in case [(a)], we obtain $w(f)< \frac{6n}{7}.$
    \item $r=1$. Let $L_{v_3}=\{w_1,w_2,w_3\}$ and
    $T'=T-\{v_1,v_2,w_1\}$. An argument similar to that described
    in [(a)] shows that $w(f)< \frac{6n}{7}$.
    \end{itemize}
\item[(d)] \quad $\deg_T(v_3)=3$.\\ Let $T'=T-T_{v_3}$. If
$f'(v_4)\neq 0$, then define $f:V(T)\rightarrow
\{0,\ldots,1+\left\lceil\frac{\Delta(T)}{2}\right\rceil\}$ by assigning
$f(x)=f'(x)$ to every $x\in V(T')$, 2 to $v_3$, 1 to the leaves at
distance 2 from $v_3$ in $T_{v_3}$, and 0 to the remaining
vertices. If $f'(v_4)= 0$, then define $f$ by assigning
$f(x)=f'(x)$ to every $x\in V(T')$, 2 to the end-stems adjacent to
$v_3$, 1 to leaf adjacent to $v_3$ if any, and 0 to the remaining
vertices. Observe that $f$ is a StRDF of $T$ of weight
$\omega(f')+|V(T_{v_3})|-1$. Since $|V(T_{v_3})|=4\;{\rm or}\;5$,
we deduce from the induction hypothesis that
$$\gamma_{StR}(T)\le \frac{6n-6|V(T_{v_3})|}{7}+|V(T_{v_3})|-1< \frac{6n}{7}.$$

\end{description}

\smallskip
{\bf Subcase 3.2.} $\deg_T(v_3)=2$.\\ By the choice of the diametral
path, we only need to consider the case in which all vertices adjacent to $v_4$ with depth 2,
have degree 2 and also, by symmetry, we may assume
$\deg(v_{k-1})=\deg(v_{k-2})=2$. Using an argument similar to
that described in Case 1, we may assume all end-stems adjacent to
$v_4$ have degree at most 3.
\begin{description}
    \item[(a)] $\deg(v_4)=2$.\\ Let $T'=T-T_{v_4}$. If
    $|V(T')|=2$, then $T=P_6$ and the result is immediate. Let
    $|V(T')|\ge 3$. If $f'(v_5)=0$, then define $f$ on $T$ by
    $f(x)=f'(x)$ for every $x\in V(T')$, $f(v_4)=f(v_2)=0$, $f(v_3)=2$
    and $f(v_1)=1$. Also, if $f'(v_5)\neq 0$, then define $f$ on $T$ by
    $f(x)=f'(x)$ for every $x\in V(T')$, $f(v_3)=f(v_1)=0$, $f(v_2)=2$
    and $f(v_4)=1$. Note that $f$ is an StRDF of $T$ of weight
    $\omega(f')+3$ and by the induction hypothesis we have
    $\gamma_{StR}(T)\le \frac{6(n-4)}{7}+3<\frac{6n}{7}$.
    \item[(b)] $\deg(v_4)\ge 3$ and $v_4$ is adjacent to an
    end-stem of degree 3, say $w$.\\ Let $w_1,w_2$ be the leaves adjacent
    to $w$ and $T'=T-\{v_1,v_2,v_3,w,w_1,w_2\}$. If
    $f'(v_4)=0$, then define $f$ on $T$ as
    $f(x)=f'(x)$ for every $x\in V(T')$, $f(w_2)=1, f(w_1)=f(v_2)=2$, $f(w)=f(v_3)=f(v_1)=0$.
    Also, if $f'(v_4)\ge 1$, then define $f$ on $T$ as
    $f(x)=f'(x)$ for every $x\in V(T')$, $f(v_1)=1, f(v_3)=f(w)=2$, $f(w_1)=f(w_2)=f(v_2)=0$.
    Again, we observe that $f$ is an StRDF of $T$ of weight at most
    $\omega(f')+5$ and, by the induction hypothesis, we have
    $\gamma_{StR}(T)\le \frac{6(n-6)}{7}+5<\frac{6n}{7}$.
    \item[(c)] $\deg(v_4)\ge 3$ and $v_4$ is adjacent to an
    end-stem of degree 2, say $w$.\\ Let $w'$ be the leaf adjacent
    to $w$. Let $T'=T-\{v_1,v_2,v_3,w,w'\}$. If
    $f'(v_4)\le 1$, then define $f$ on $T$ as
    $f(x)=f'(x)$ for every $x\in V(T')$, $f(w')=f(v_2)=2$, $f(w)=f(v_3)=f(v_1)=0$.
    Also, if $f'(v_4)\ge 2$, then define $f$ on $T$ by
    $f(x)=f'(x)$ for every $x\in V(T')-\{v_4\}$, $f(v_4)=f'(v_4)+1$, $f(w')=1, f(v_2)=2$, $f(w)=f(v_3)=f(v_1)=0$.
    Clearly, $f$ is an StRDF of $T$ of weight at most
    $\omega(f')+5$ and, by the induction hypothesis, we have
    $\gamma_{StR}(T)\le \frac{6(n-6)}{7}+5<\frac{6n}{7}$.
    \item[(d)] $\deg(v_4)\ge 3$ and there is a path $v_4w_3w_2w_1$ in $T$ such that $w_3\not\in
    \{v_3,v_5\}$.\\Hence, we must have $\delta(w_3)=\delta(w_2)=2$ and
    $\delta(w_1)=1$. Let $T'=T-\{v_i,w_i\mid 1\le i\le 2\}$. If
    $f'(v_4)\le 1$, then define $f$ on $T$ as
    $f(x)=f'(x)$ for every $x\in V(T')$, $f(w_2)=f(v_2)=2$, $f(w_3)=f(v_3)=f(v_1)=f(w_1)=0$.
    Also, if $f'(v_4)\ge 2$, then define $f$ on $T$ as
    $f(x)=f'(x)$ for $x\in V(T')-\{v_4\}$, $f(v_4)=f'(v_4)+1$, $f(w_2)=f(v_2)=2$, $f(w_3)=f(v_3)=f(v_1)=f(w_1)=0$.
    Notice that $f$ is an StRDF of $T$ of weight at most
    $\omega(f')+4$ and, by the induction hypothesis, we have
    $\gamma_{StR}(T)\le \frac{6(n-5)}{7}+4<\frac{6n}{7}$.
    \item[(e)] $\deg(v_4)=3$ and $v_4$ is adjacent to a leaf, say $w$.\\Hence, clearly ${\rm diam}(T)\ge 6$.
    Let $T'=T-T_{v_4}$. If
    $f'(v_4)\ge 1$, then define $f$ on $T$ by
    $f(x)=f'(x)$ for every $x\in V(T')$, $f(v_4)=f(v_2)=2$, $f(w)=f(v_3)=f(v_1)=0$.
    Also, if $f'(v_4)=0$, then define $f$ on $T$ by
    $f(x)=f'(x)$ for every $x\in V(T')$, $f(w)=f(v_2)=2$, $f(v_4)=f(v_3)=f(v_1)=0$.
    Obviously $f$ is an StRDF of $T$ of weight at most
    $\omega(f')+4$ and, by the induction hypothesis, we have
    $\gamma_{StR}(T)\le \frac{6(n-5)}{7}+4<\frac{6n}{7}$.
    \item[(f)] $\deg(v_4)\ge 4$ and every neighbor of $v_4$ but $v_3,v_5$ is a leaf.\\Clearly ${\rm diam}(T)\ge 6$.
    Let $T'=T-T_{v_4}$ and define $f$ on $T$ by
    $f(x)=f'(x)$ for every $x\in V(T')$, $f(v_4)=1+\left\lceil\frac{\deg(v_4)}{2}\right\rceil$, $f(v_2)=2$ and $f(x)=0$ otherwise.
    Note that $f$ is an StRDF of $T$ of weight
    $\omega(f')+3+\left\lceil\frac{\deg(v_4)}{2}\right\rceil$. Using the induction hypothesis we
    can check that $\gamma_{StR}(T)<\frac{6n}{7}$.
\end{description}
This completes the proof.
\end{proof}

Let $S(K_{1,3})$ (the star $K_{1,3}$ with all its edges subdivided) be rooted in its center $v$ and let $F^p_{m}$ consist of all the rooted product graphs $T\circ_v S(K_{1,3})$, where $T$ is any tree on $m$ vertices (see Figure \ref{fig1} for an example). We notice that, if the graph $H=S(K_{1,3})$  is an induced subgraph in a graph $G$, and its noncentral vertices have no neighbors outside $H$ in $G$, then any StRDF must put total weight
at least 6 on the vertices of $H$. For the case of trees $T\in F_m^p$, they contain $m$ disjoint induced subgraphs isomorphic to $S(K_{1,3})$ satisfying the situation mentioned above. So, $\gamma_{StR}(T)\ge 6|V(T)|/7$ for each $T\in F_m^p$. Clearly, the trees belonging to $F^p_{m}$ are dependent of the rooted tree $T'$ used in the rooted product $T'\circ_v S(K_{1,3})$, and this
help us to characterize the trees achieving equality in Theorem \ref{tree1}.

\begin{figure}[ht]
\def\emline#1#2#3#4#5#6{%
\put(#1,#2){\special{em:moveto}}%
\put(#4,#5){\special{em:lineto}}}
\def\newpic#1{}
\begin{center} \unitlength 1mm
\special{em:linewidth 0.4pt} \linethickness{0.6pt}
\begin{picture}(140.33,80.67)
\put(30.00,80.00){\line(-1,-3){6.50}}
\put(30.00,80.00){\line(1,-3){6.50}}
\put(30.00,80.00){\line(-0,-2){20.50}}
\put(30.00,60.00){\circle*{1.33}}
\put(30.00,70.00){\circle*{1.33}}
\put(23.33,60.00){\circle*{1.33}}
\put(26.67,70.00){\circle*{1.33}}
\put(33.33,70.00){\circle*{1.33}}
\put(36.67,60.00){\circle*{1.33}}
\put(55.00,80.00){\line(-1,-3){6.50}}
\put(55.00,80.00){\line(1,-3){6.50}}
\put(55.00,80.00){\line(-0,-2){20.50}}
\put(55.00,60.00){\circle*{1.33}}
\put(55.00,70.00){\circle*{1.33}}
\put(48.33,60.00){\circle*{1.33}}
\put(51.67,70.00){\circle*{1.33}}
\put(58.33,70.00){\circle*{1.33}}
\put(61.67,60.00){\circle*{1.33}}
\put(80.00,80.00){\line(-1,-3){6.50}}
\put(80.00,80.00){\line(1,-3){6.50}}
\put(80.00,80.00){\line(-0,-2){20.50}}
\put(80.00,60.00){\circle*{1.33}}
\put(80.00,70.00){\circle*{1.33}}
\put(73.33,60.00){\circle*{1.33}}
\put(76.67,70.00){\circle*{1.33}}
\put(83.33,70.00){\circle*{1.33}}
\put(86.67,60.00){\circle*{1.33}}
\put(105.00,80.00){\line(-1,-3){6.50}}
\put(105.00,80.00){\line(1,-3){6.50}}
\put(105.00,80.00){\line(-0,-2){20.50}}
\put(105.00,60.00){\circle*{1.33}}
\put(105.00,70.00){\circle*{1.33}}
\put(98.33,60.00){\circle*{1.33}}
\put(101.67,70.00){\circle*{1.33}}
\put(108.33,70.00){\circle*{1.33}}
\put(111.67,60.00){\circle*{1.33}}
\put(130.00,80.00){\line(-1,-3){6.50}}
\put(130.00,80.00){\line(1,-3){6.50}}
\put(130.00,80.00){\line(-0,-2){20.50}}
\put(130.00,60.00){\circle*{1.33}}
\put(130.00,70.00){\circle*{1.33}}
\put(123.33,60.00){\circle*{1.33}}
\put(126.67,70.00){\circle*{1.33}}
\put(133.33,70.00){\circle*{1.33}}
\put(136.67,60.00){\circle*{1.33}}
\put(30.00,80.00){\line(1,0){100.00}}
\put(30.00,80.00){\circle*{1.33}}
\put(55.00,80.00){\circle*{1.33}}
\put(80.00,80.00){\circle*{1.33}}
\put(105.00,80.00){\circle*{1.33}}
\put(130.00,80.00){\circle*{1.33}}
\end{picture}
\end{center}
\vspace*{-2cm}
\vspace{-4 cm} \caption{A member of $F_{5}^p$} \label{fig1}
\end{figure}
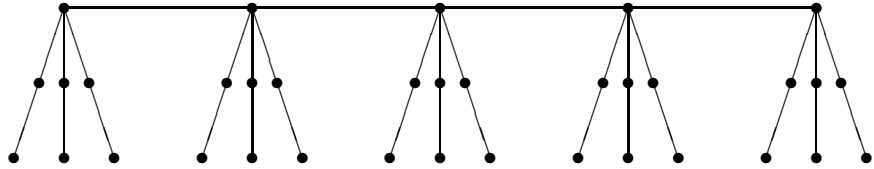

\begin{theorem}\label{tree2}
Let $T$ be an $n$-vertex tree. Then $\gamma_{StR}(T)=6n/7$ if
and only if $T\in F^p_{m}$.
\end{theorem}

\begin{proof}
We have seen that if an induced subgraph $H$ of $G$ is isomorphic
to $S(K_{1,3})$, and its noncentral vertices have no neighbors
outside $H$ in $G$, then every StRDF of $G$ must put weight at
least 6 on $V(H)$. Since every tree $T\in F^p_{m}$ has a vertex
partition of $m$ sets inducing such a subgraphs, the weight at least 6 is needed on every set of such
partition. Moreover, it is easy to find a $\gamma_{StR}(T)$-function of weight $6n/7$, which leads to the equality.

To prove that equality requires this structure, we examine the
proof of Theorem \ref{tree1} more closely. We proceed by induction
on $n$. In the base cases, ${\rm diam}(T)\le 3$ and Cases 1 and 2,
we produce an StRDF of weight less than $6n/7$. In Case 3 with
diameter 4, equality requires $T=S(K_{1,3})$. Let ${\rm
diam}(T)\ge 5$ and let $v_1v_2\ldots v_k$ be a diametral path in
$T$. Root $T$ at $v_k$ and let $T'=T-T_{v_3}$. Since
$\gamma_{StR}(T)=6n/7$, we deduce that $T_{v_3}=S(K_{1,3})$ and
$\gamma_{StR}(T')=\frac{6|V(T_1)|}{7}$. By the induction
hypothesis, $V(T_1)$ can be partitioned into sets inducing
$S(K_{1,3})$ such that the subgraph induced by the central
vertices of these subdivision stars is connected. Suppose
$\{v,u_1,u_2,u_3,w_1,w_2,w_3\}$ is the partition set inducing
$S(K_{1,3})$ with central vertex $v$ and leaves  $u_1,u_2,u_3$
containing $v_4$ in which $w_i$ is the support vertex of $u_i$ for
 each $i$. We claim that $v_{4}=v$. Otherwise, we may assume without loss
of generality that $v_{4}\in \{u_1,u_2,u_3,w_1,w_2,w_3\}$. If
$v_4\in \{u_1,u_2,u_3\}$, then define $f:V(T)\rightarrow
\{0,1,2,3\}$ by $f(v_4)=f(w_1)=f(w_2)=f(w_3)=0, f(v)=3, f(x)=1$
for every $x\in \{u_1,u_2,u_3\}-\{v_4\}$ and let $f$ assign 3 to all
other central vertices, 0 to all neighbors of central vertices and
1 to all leaves in each partition sets inducing $S(K_{1,3})$. If
$v_4\in \{w_1,w_2,w_3\}$, then define $f:V(T)\rightarrow
\{0,1,2,3\}$ by $f(w_1)=f(w_2)=f(w_3)=0, f(v)=2, f(x)=1$ for $x\in
\{u_1,u_2,u_3\}$ and let $f$ assign 3 to all other central
vertices, 0 to all neighbors of central vertices and 1 to all
leaves in each partition sets inducing $S(K_{1,3})$. It is easy to
see that in each case, $f$ is a StRDF of $T$ with weight less than
$6n/7$ which is a contradiction. Thus $v_4$ is the central vertex
of the subdivision vertex $S(K_{1,3})$ and the proof is completed.
\end{proof}

 In the sequel, we characterize all trees $T$ of order $n\ge 3$ with $\gamma_{StR}(G)=n-1.$

\begin{lemma}\label{distances-3}
Let $G$ be a connected graph of order $n$ satisfying one of the following statement:
\begin{enumerate}
  \item $G$ has two vertices $x$ and $y$  of degree at least 2 with $d_G(x,y)\ge 3$.
  \item $G$ has two adjacent vertices $x$ and $y$ of degree at least three with $N(x)\cap N(y)=\emptyset$
  \item $G$ has two non-adjacent vertices $x$ and $y$ of degree at least three with $|N(x)\cap N(y)|\le 1$.
\end{enumerate}
Then $\gamma_{StR}(G)\le n-2$.
\end{lemma}

\begin{proof}
Assume that $\{x_1,x_2\}\subseteq  N(x)\setminus N[y]$ and $\{y_1,y_2\}\subseteq N(y)\setminus N[x]$.
By assumption $\{x_1,x_2\}\cap \{y_1,y_2\}=\emptyset$. Then the function $f$
given by $f(u)=0$ for $u\in \{x_1,x_2,y_1,y_2\}$, $f(x)=f(y)=2$ and $f(u)=1$ otherwise,
is a strong Roman dominating function of weight $w(f)=n-6+4=n-2$ and hence $\gamma_{StR}(G)\le n-2$.
\end{proof}

Let $\mathcal{T}$ be the family of trees consisting of the paths $P_3,P_4,P_5$ and the following four trees.

\begin{figure}[ht]
\centering
\begin{tikzpicture}[scale=.5, transform shape]
\node [draw, shape=circle,fill=black] (x1) at  (0,0) {};
\node [draw, shape=circle,fill=black] (y1) at  (2,0) {};
\node [draw, shape=circle,fill=black] (z1) at  (4,0) {};
\node [draw, shape=circle,fill=black] (x2) at  (7,0) {};
\node [draw, shape=circle,fill=black] (y2) at  (9,0) {};
\node [draw, shape=circle,fill=black] (z2) at  (11,0) {};
\node [draw, shape=circle,fill=black] (x3) at  (14,0) {};
\node [draw, shape=circle,fill=black] (y3) at  (16,0) {};
\node [draw, shape=circle,fill=black] (z3) at  (18,0) {};
\node [draw, shape=circle,fill=black] (x4) at  (21,0) {};
\node [draw, shape=circle,fill=black] (y4) at  (23,0) {};
\node [draw, shape=circle,fill=black] (z4) at  (25,0) {};
\node [draw, shape=circle,fill=black] (y22) at  (9,-2.5) {};
\node [draw, shape=circle,fill=black] (x33) at  (14,-2.5) {};
\node [draw, shape=circle,fill=black] (z33) at  (18,-2.5) {};
\node [draw, shape=circle,fill=black] (x44) at  (21,-2.5) {};
\node [draw, shape=circle,fill=black] (y44) at  (23,-2.5) {};
\node [draw, shape=circle,fill=black] (z44) at  (25,-2.5) {};
\node [draw, shape=circle,fill=black] (c1) at  (2,3) {};
\node [draw, shape=circle,fill=black] (c2) at  (9,3) {};
\node [draw, shape=circle,fill=black] (c3) at  (16,3) {};
\node [draw, shape=circle,fill=black] (c4) at  (23,3) {};
\node [scale=1.4] at (2,-4.5) {(a)};
\node [scale=1.4] at (9,-4.5) {(b)};
\node [scale=1.4] at (16,-4.5) {(c)};
\node [scale=1.4] at (23,-4.5) {(d)};

\draw(x1)--(c1)--(y1);
\draw(c1)--(z1);

\draw(y22)--(y2)--(c2)--(x2);
\draw(c2)--(z2);

\draw(x33)--(x3)--(c3)--(z3)--(z33);
\draw(c3)--(y3);

\draw(x44)--(x4)--(c4)--(z4)--(z44);
\draw(c4)--(y4)--(y44);
\end{tikzpicture}
\caption{The trees of the family $\mathcal{T}$.}\label{arboles}
\end{figure}
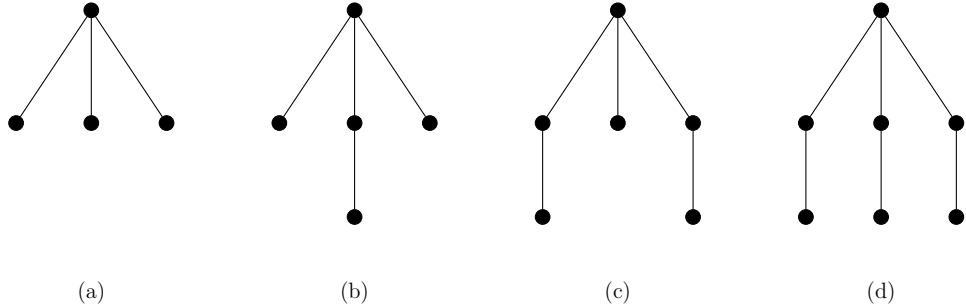

\begin{theorem}\label{characterize}
Let $T$ be a tree of order $n\ge 3$. Then $\gamma_{StR}(T)=n-1$ if and only if
$T\in \mathcal{T}$.
\end{theorem}

\begin{proof}
If $T\in \mathcal T$, then it is easy to see that $\gamma_{StR}(T)=n-1$.

Conversely, let $\gamma_{StR}(T)=n-1$. By Proposition
\ref{prop-2-delta} and Theorem \ref{tree1}, we deduce that
$\Delta(T)\le 3$ and $n\le 7$. If $\Delta(T) = 2$, then it follows
from Proposition \ref{path} and Observation \ref{le2} that $T
\cong P_q$ for $q=3,4,5$. Now, let $\Delta(T)= 3$. We deduce from
Proposition \ref{diam} and Lemma \ref{distances-3} that ${\rm
diam}(T)\le 4$ and $T$ has a unique vertex of degree 3.  This
implies that $T$ is one one the trees in Figure \label{arboles}
and so $T\in \mathcal{T}$.
\end{proof}

We conclude this section with an open problem.

\smallskip
\noindent {\bf Conjecture.} For any connected graph $G$ of order
$n\ge 3$, $$\gamma_{StR}(G)\le 6n/7.$$

Moreover, we notice that if such abound will be true, then the equality
holds if and only if $G$ is obtained as the rooted product graph $G\circ_v S(K_{1,3})$, where $G$ is any connected graph and $S(K_{1,3})$ is rooted in its center $v$.

\section{Realizability for trees}

It is clear that the maximum number of legions which are necessary
to defend a graph $G$, under the strong Roman domination strategy,
is the order of the graph. In this sense, as usual when a new
parameter is introduced, it would be interesting to know if there
exist graphs of order $n$, achieving all the possible suitable
values for the strong Roman domination number. That is, in
concordance with Proposition \ref{gminimo}, all the integer
numbers in the interval $\left\{\left\lceil \frac{n+1}{2}
\right\rceil,\ldots,n\right\}$. Equivalently, we should deal with
the problem of realization for the strong Roman domination. That
is, given two positive integers $n,p$ such that $\left\lceil
\frac{n+1}{2} \right\rceil\le p\le n$: Is there a graph of order
$n$ and  strong Roman domination number equal to $p$? Next we
partially solve this problem in concordance with the conjecture presented at the end of the section above. To this end, we need the following lemma, where we give the strong Roman domination number of spiders.

\begin{lemma}\label{L}
If $T$ is a spider obtained from $K_{1,t}\;(t\ge 2)$ by subdividing $q$ edges $(0\le q\le t)$, then $$\gamma_{StR}(T)=1+q+\left\lceil\frac{t}{2}\right\rceil.$$
\end{lemma}

\begin{proof}
If $t=2$, then $T$ is a path and the result is immediate by
Theorem \ref{path} and Observation \ref{le2}. If $t=3$, then the
result follows from Theorem \ref{characterize}. Let $t\ge 4$. By
Corollary \ref{star}, we may assume that $q\ge 1$. Let $v$ be the
central vertex of $K_{1,t}$ and let $N(v)=\{v_1,\ldots,v_t\}$.
Suppose $u_i$ is the subdivision vertex of the edge $vv_i$ for
$1\le i\le q$. Defined $f:V(T)\rightarrow
\{0,1,\ldots,1+\left\lceil\frac{t}{2}\right\rceil\}$ by
$f(v)=1+\left\lceil\frac{t}{2}\right\rceil$, $f(v_i)=1$ for $1\le i\le q$ and
$f(x)=0$ otherwise. It is easy to see that $f$ is a strong Roman
dominating function of $T$ of weight $1+q+\left\lceil\frac{t}{2}\right\rceil$
and hence $\gamma_{StR}(T)\le 1+q+\left\lceil\frac{t}{2}\right\rceil$.

Now we show that $\gamma_{StR}(T)\ge1+q+\left\lceil\frac{t}{2}\right\rceil$.
Let $f$ be a $\gamma_{StR}(T)$-function such that $f(v)$ is as
large as possible. If $f(v)\le 1$, then clearly $\omega(f)\ge n-1$
that leads to a  contradiction with Theorem \ref{characterize}.
Let $f(v)\ge 2$. If $f(u_i)\ge 1$ for some $1\le i\le q$, then the
function $g:V(T)\rightarrow
\{0,1,\ldots,1+\left\lceil\frac{t}{2}\right\rceil\}$ defined by $g(v)=f(v)+1,
g(u_i)=0, g(v_i)=1$ and $g(x)=f(x)$ otherwise, is a
$\gamma_{StR}(T)$-function that contradicts the choice of $f$. If
$f(v_j)\ge 1$ for some $q+1\le j\le t$, then the function
$g:V(T)\rightarrow \{0,1,\ldots,1+\left\lceil\frac{t}{2}\right\rceil\}$
defined by $g(v)=f(v)+1, g(v_j)=0$ and $g(x)=f(x)$ otherwise, is a
$\gamma_{StR}(T)$-function that contradicts the choice of $f$
again. Thus $f(v_i)=0$ for each $1\le i\le t$ implying that
$f(v)\ge 1+\left\lceil\frac{t}{2}\right\rceil$ and $f(v_i)=1$ for $1\le i\le
q$. Hence $\gamma_{StR}(T)\ge1+q+\left\lceil\frac{t}{2}\right\rceil$ and the
proof is complete.
\end{proof}

We also need the next result, where we compute the strong Roman domination number of some family of graphs $\widetilde{\mathcal{F}}$ defined as follows. A graph $G_n(q,j,l)\in \widetilde{\mathcal{F}}$ if and only if $G_n(q,j,l)$ is obtained as the rooted product graph $G(\mathcal{H})$, where $G=v_1v_2...v_{q+1}$ is a path on $q+1$ vertices and the ordered family $\mathcal{H}=\{S_{j+l,j},S_{3,3},S_{3,3},...,S_{3,3}\}$ is formed by one (wounded or healthy) spider $S_{j+l,j}$ and $q$ healthy spiders $S_{3,3}$ having their roots in their centers. See Figure \ref{family-F-tilde} for an example.

\begin{figure}[ht]
\centering
\begin{tikzpicture}[scale=.5, transform shape]
\node [draw, shape=circle,fill=black] (x1) at  (0,0) {};
\node [draw, shape=circle,fill=black] (y1) at  (2,0) {};
\node [draw, shape=circle,fill=black] (z1) at  (4,0) {};
\node [draw, shape=circle,fill=black] (y12) at  (-2,0) {};
\node [draw, shape=circle,fill=black] (z12) at  (-4,0) {};
\node [draw, shape=circle,fill=black] (x2) at  (7,0) {};
\node [draw, shape=circle,fill=black] (y2) at  (9,0) {};
\node [draw, shape=circle,fill=black] (z2) at  (11,0) {};
\node [draw, shape=circle,fill=black] (x3) at  (14,0) {};
\node [draw, shape=circle,fill=black] (y3) at  (16,0) {};
\node [draw, shape=circle,fill=black] (z3) at  (18,0) {};
\node [draw, shape=circle,fill=black] (x4) at  (21,0) {};
\node [draw, shape=circle,fill=black] (y4) at  (23,0) {};
\node [draw, shape=circle,fill=black] (z4) at  (25,0) {};
\node [draw, shape=circle,fill=black] (y111) at  (2,-2.5) {};
\node [draw, shape=circle,fill=black] (z111) at  (4,-2.5) {};
\node [draw, shape=circle,fill=black] (x211) at  (7,-2.5) {};
\node [draw, shape=circle,fill=black] (y22) at  (9,-2.5) {};
\node [draw, shape=circle,fill=black] (z211) at  (11,-2.5) {};
\node [draw, shape=circle,fill=black] (x33) at  (14,-2.5) {};
\node [draw, shape=circle,fill=black] (y311) at  (16,-2.5) {};
\node [draw, shape=circle,fill=black] (z33) at  (18,-2.5) {};
\node [draw, shape=circle,fill=black] (x44) at  (21,-2.5) {};
\node [draw, shape=circle,fill=black] (y44) at  (23,-2.5) {};
\node [draw, shape=circle,fill=black] (z44) at  (25,-2.5) {};
\node [draw, shape=circle,fill=black] (c1) at  (0,3) {};
\node [draw, shape=circle,fill=black] (c2) at  (9,3) {};
\node [draw, shape=circle,fill=black] (c3) at  (16,3) {};
\node [draw, shape=circle,fill=black] (c4) at  (23,3) {};

\draw(x1)--(c1)--(y1);
\draw(c1)--(z1);

\draw(y12)--(c1)--(z12);
\draw(c1)--(c2)--(c3)--(c4);
\draw(y1)--(y111);
\draw(z1)--(z111);
\draw(x2)--(x211);
\draw(z2)--(z211);
\draw(y3)--(y311);

\draw(y22)--(y2)--(c2)--(x2);
\draw(c2)--(z2);

\draw(x33)--(x3)--(c3)--(z3)--(z33);
\draw(c3)--(y3);

\draw(x44)--(x4)--(c4)--(z4)--(z44);
\draw(c4)--(y4)--(y44);
\end{tikzpicture}
\caption{The graph $G_{29}(3,2,3)$ of the family $\widetilde{\mathcal{F}}$.}\label{family-F-tilde}
\end{figure}
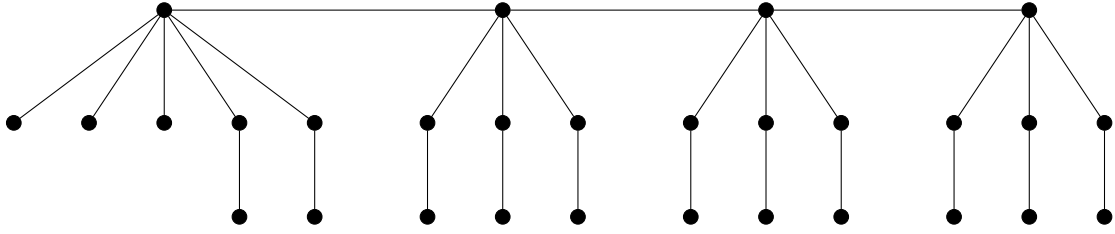

\begin{lemma}\label{G_n(q,j,l)}
For any graph $G_n(q,j,l)\in \widetilde{\mathcal{F}}$, $\gamma_{StR}(G_n(q,j,l))=6q+\left\lceil\frac{j+l}{2}\right\rceil+j+1$.
\end{lemma}

\begin{proof}
By Lemma \ref{L} we have that $\gamma_{StR}(S_{j+l,j})=1+j+\left\lceil\frac{j+l}{2}\right\rceil$. Also, by Lemma \ref{tree2} we know that the subgraph $G'$ of $G_n(q,j,l)$ induced by the $q$ healthy spiders $S_{3,3}$ satisfies that $\gamma_{StR}(G')=6q$. Clearly, a function $f$ in $G_n(q,j,l)$ obtained from a combination of a $\gamma_{StR}(S_{j+l,j})$-function $f_1$ and a $\gamma_{StR}(G')$-function $f_2$, is a strong Roman dominating function in $G_n(q,j,l)$. Thus, we have that $\gamma_{StR}(G_n(q,j,l))\le w(f_1)+w(f_2)=6q+\left\lceil\frac{j+l}{2}\right\rceil+j+1$.

On the other hand, an argument similar to that described in the proof of Lemma
\ref{L} can be used to obtain the lower bound. Since the induced subgraph $S_{j+l,j}$ of $G_n(q,j,l)$ is isomorphic
to a spider, and its noncentral vertices have no neighbors
outside of $S_{j+l,j}$, every StRDF of $G_n(q,j,l)$ must put weight at least $1+q+\left\lceil\frac{j+l}{2}\right\rceil$ on $V(S_{j+l,j})$. Similarly, every StRDF of $G_n(q,j,l)$ must put weight at least $6q$ in the subgraph $G'$ of $G_n(q,j,l)$ induced by the $q$ healthy spiders $S_{3,3}$. Therefore, $\gamma_{StR}(G_n(q,j,l))\ge 6q+\left\lceil\frac{j+l}{2}\right\rceil+j+1$ and the proof is completed.
\end{proof}

Once given the value for the strong Roman domination of graphs of the family $\widetilde{\mathcal{F}}$ we are able to present our realizability result for trees.

\begin{theorem}
Let $n,p$ be any two integers such that $n\ge 3$ and $\left\lceil\frac{n+1}{2}\right\rceil\le p\le \left\lfloor\frac{6n}{7}\right\rfloor$. Then, there exists a graph $G$ of order $n$ with $\gamma_{StR}(G)=p$.
\end{theorem}

\begin{proof}
If $n=3$, then $\left\lceil \frac{n+1}{2}\right\rceil\leq\left\lfloor \frac{6n}{7}\right\rfloor =2$. So $p=2$. Hence, a possible graph for this case would be the path $P_3$. If $n=4$, then $\left\lceil\frac{n+1}{2}\right\rceil =\left\lfloor \frac{6n}{7}\right\rfloor =3$. Thus, it must happen $p=3$ and a possibility for the realization of a graph would be the star graph $S_{1,3}$.

From now on we consider $n\ge 5$ and for these cases, we shall use the graphs of the family $\widetilde{\mathcal{F}}$ for some specific values of $q,j,l$,  satisfying that $0\le q \le \left\lfloor\frac{n}{7}\right\rfloor$, $0\le j\le 4$ and $j+l\ge 3$ (we recall that the condition $0\le j\le 4$ is introduced for our purposes on the existence of a graph realizing the values of $n$ and $p$). From Lemma \ref{G_n(q,j,l)}, we know that a graph $G_n(q,j,l)\in \widetilde{\mathcal{F}}$ has order $n$ and $\gamma_{StR}(G_n(q,j,l))=6q+\left\lceil\frac{j+l}{2}\right\rceil+j+1$. Thus, given the integers $n,p$ as described in the statement, it only remains to find the suitable values of $q,j,l$ in terms of $n,p$, or equivalently, to obtain the solutions of the following systems of equations.
\begin{equation}\label{system}
\left\{\begin{array}{l}
    n=7q+2j+l+1 \\
    p=6q+\left\lceil\frac{j+l}{2}\right\rceil+j+1
  \end{array}\right.
\end{equation}
We proceed with the solution of the system above, for $q$ and $l$, according to the parity of the number $j+l$ and using the fact that $0\le j\le 4$. If $j+l$ is even, then the system (\ref{system}) becomes
\begin{equation}\label{system-even}
\left\{\begin{array}{l}
    n=7q+2j+l+1 \\
    2p=12q+3j+l+2
  \end{array}\right.
\end{equation}
whose solution (for $q$ and $l$) in terms of $n,p,j$ is given by
\begin{equation}\label{solution-even}
q=\frac{2p-n-j-1}{5}\,\mbox{ and }\, l=\frac{12n-14p-3j+2}{5}.
\end{equation}
If $j+l$ is odd, then the system (\ref{system}) becomes
\begin{equation}\label{system-odd}
\left\{\begin{array}{l}
    n=7q+2j+l+1 \\
    2p=12q+3j+l+3
  \end{array}\right.
\end{equation}
whose solution (for $q$ and $l$) in terms of $n,p,j$ is given by
\begin{equation}\label{solution-odd}
q=\frac{2p-n-j-2}{5}\,\mbox{ and }\, l=\frac{12n-14p-3j+9}{5}.
\end{equation}
We check now that these solutions are giving suitable values for $q,l$, according to the fact that $0\le j\le 4$, in order to construct the graph $G_n(q,j,l)$. That is, we must check that integer solutions to the systems in (\ref{system-even}) and (\ref{system-odd}) are possible. Consider the solution in (\ref{solution-even}). There we have that $2p-n=5q+j+1$. Since $p\ge \frac{n+1}{2}$, it follows that $2p-n\ge 1$. Thus, $2p-n$ can always be represented as $5q+j+1$ for some integer $q$ and some $j\in \{0,...,4\}$. Let $q',j'$ obtained in this way. We must check that $q',j'$ lead to an integer solution for $l$ in (\ref{solution-even}). That is,
\begin{align*}
l&=\frac{12n-14p-3j'+2}{5}\\
&=n-\frac{7(2p-n)+3j'-2}{5}\\
&=n-\frac{7(5q'+j'+1)+3j'-2}{5}\\
&=n-7q'+2j'+1.
\end{align*}
Thus, $l$ is an integer number and the construction of $G_n(q,j,l)$ is possible. An analogous process leads to integer solutions of the system in  (\ref{system-odd}).

Now, in order to complete the proof, we observe that the solutions for $q$ in (\ref{solution-even}) and (\ref{solution-odd}), lead to $p\ge \frac{n+1}{2}$ and $p\ge \frac{n+2}{2}$, respectively, since $j,q\ge 0$. So, it must happen $p\ge \left\lceil\frac{n+1}{2}\right\rceil$. It remains to check only the consequences of the solution for $l$. To this end, we consider the following cases.\\

\noindent Case 1: $j=0$. If $l$ is even, then $j+l$ is even. Since $j+l\ge 3$ (by assumption) and $l$ is even, we obtain that $l\ge 4$. Hence, from (\ref{solution-even}) we have $12n-4p+2\ge 20$, which leads to $p\le \frac{6n-9}{7}<\frac{6n}{7}$. If $l$ is odd, then $j+l$ is odd and $l\ge 3$. From (\ref{solution-even}) we have $12n-4p+2\ge 15$, which leads to $p\le \frac{6n-9}{7}<\frac{6n}{7}$.\\

\noindent Case 2: $j\ne 0$. If $j+l$ is even, then we must consider the solutions in (\ref{solution-even}). If $j$ is even, then $l$ is even and $l\ge 2$ (since $j+l\ge 3$). So, from the solution for $l$ in (\ref{solution-even}) we have that $12n-14p-3j+2\ge 10$, which means that $p\le \frac{12n-3j-8}{14}\le \frac{12n-14}{14}=\frac{6n}{7}-1<\frac{6n}{7}$. If $j$ is odd, then $l$ is odd and $l\ge 3$ (again since $j+l\ge 3$). Thus, the solution for $l$ in (\ref{solution-even}) leads to $12n-14p-3j+2\ge 15$, which gives $p\le \frac{12n-3j-13}{14}\le \frac{12n-16}{14}<\frac{6n}{7}-1<\frac{6n}{7}$. Finally, if $j+l$ is odd, then a similar process as above, by considering the parity of $j$, gives that always $p\le \frac{6n}{7}$, which completes the proof.
\end{proof}

As a conclusion of the proof above, given two integers $n,p$ such that $n\ge 3$ and $\left\lceil\frac{n+1}{2}\right\rceil\le p\le \left\lfloor\frac{6n}{7}\right\rfloor$, the construction of a graph $G$ of order $n$ with $\gamma_{StR}(G)=p$ is made in the following way.
\begin{itemize}
\item Find a value $q'$ such that $2p-n=5q'+j'+1$ for some $j'\in \{0,...,4\}$.
\item Compute the value $l'=n-7q'+2j'+1$.
\item Construct the desired graph $G$ as the graph $G_n(q',j',l')$.
\end{itemize}

\section*{Conclusions}

In this article we have introduced a new invariant related to domination in graphs, called the strong Roman domination number and denoted $\gamma_{StR}(G)$, for any graph $G$. We have proved that the decision problem regarding the existence of a strong Roman dominating function of minimum cardinality belongs to the NP-complete complexity class. In concordance with this fact, we have obtained several lower and upper bounds for $\gamma_{StR}(G)$ of any connected graph $G$. Such bounds regard some parameters in graphs like for instance, the order, the diameter, the girth, among others. We have also obtained an interesting upper bound on the strong Roman domination number of trees, namely $\gamma_{StR}(T)\le \left\lfloor\frac{6n}{7}\right\rfloor$ for any tree $T$ of order $n$, and have characterized the families of trees achieving such a bound. In concordance with such bound, and with a lower bound involving the order of any graph $G$, that is $\gamma_{StR}(G)\ge \left\lceil\frac{n+1}{2}\right\rceil$, we have presented a realizability result which involves all the possible values between such bounds. Specifically, we have proved that given two integers $n,p$ such that $n\ge 3$ and $\left\lceil\frac{n+1}{2}\right\rceil\le p\le \left\lfloor\frac{6n}{7}\right\rfloor$, there exists a tree $T$ of order $n$ with $\gamma_{StR}(T)=p$.

\end{document}